\documentclass[12pt,reqno]{article}
\usepackage{fullpage}
\usepackage{float}
\usepackage{upgreek}
\usepackage{booktabs, enumerate,float}
\usepackage{amsthm,amsmath,amssymb}
\usepackage{tikz}
\usepackage{subfigure}
\usepackage{tabularx}
\usetikzlibrary{arrows,matrix,positioning}
\usepackage{float,multirow,multicol}
\usepackage{graphicx,tipa}
\usepackage[export]{adjustbox}

\definecolor{webgreen}{rgb}{0,.5,0}
\definecolor{webbrown}{rgb}{.6,0,0}
\usepackage[colorlinks=true,
linkcolor=webgreen,
filecolor=webbrown,
citecolor=webgreen]{hyperref}

\newcommand{\arxiv}[1]{\href{http://arxiv.org/abs/#1}{\texttt{arXiv:#1}}}
\newcommand{\seqnum}[1]{\href{http://oeis.org/#1}{\underline{#1}}}
\newcommand{\Figs}[1]{\hyperref[#1]{Figure~\ref*{#1}}}
\newcommand{\Tabs}[1]{\hyperref[#1]{Table~\ref*{#1}}}
\newcommand{\Cor}[1]{\hyperref[#1]{Corollary~\ref*{#1}}}
\newcommand{\Prop}[1]{\hyperref[#1]{Proposition~\ref*{#1}}}
\newcommand{\Thm}[1]{\hyperref[#1]{Theorem~\ref*{#1}}}
\newcommand{\Lem}[1]{\hyperref[#1]{Lemma~\ref*{#1}}}
\newcommand{\Sec}[1]{\hyperref[#1]{section~\ref*{#1}}}
\newcommand{\Subsec}[1]{\hyperref[#1]{subsection~\ref*{#1}}}
\newcommand{\Rem}[1]{\hyperref[#1]{Remark~\ref*{#1}}}
\everymath{\displaystyle}

\theoremstyle{plain}
\newtheorem{theorem}{Theorem}

\newtheorem{corollary}[theorem]{Corollary}
\newtheorem{proposition}[theorem]{Proposition}

\theoremstyle{definition}
\newtheorem{definition}[theorem]{Definition}
\newtheorem{example}[theorem]{Example}
\newtheorem{notation}[theorem]{Notation}

\theoremstyle{remark}
\newtheorem{remark}[theorem]{Remark}

\title{\bf  Statistics on some classes of knot shadows}
\author{Franck Ramaharo\\
\small D\'epartement de Math\'ematiques et Informatique\\[-0.8ex]
\small Universit\'e d'Antananarivo\\[-0.8ex] 
\small 101 Antananarivo, Madagascar\\
\small\href{mailto:franck.ramaharo@gmail.com}{\tt franck.ramaharo@gmail.com}\\
}

\date{\small\today\\}

\begin{document}

\maketitle

\begin{abstract}
The present paper is concerned with the enumeration of the state diagrams for some classes of  knot shadows endowed with the usual connected sum operation. We    focus on shadows that are recursively generated by knot shadows with up to $ 3 $ crossings, and for which the enumeration problem is solved with the help of  generating polynomials.

\end{abstract}

\section{Introduction}
Let a mathematical knot be identified with its regular projection onto the sphere $ S^2 $. The corresponding representation, called \textit{shadow}, is a planar quadrivalent diagram without the usual under/over information at every crossing \cite{Denton,Hanaki,HHJJMR,MRS}. We can split each crossing of the diagram in one of two ways as shown in \Figs{Fig:split}.
\begin{figure}[H]
\centering
\hspace*{\fill}
\subfigure[Type $ 0 $ split]{\includegraphics[width=0.2\linewidth]{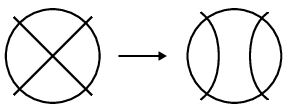}}\hfill%
\subfigure[Type $ 1 $ split]{\includegraphics[width=0.2\linewidth]{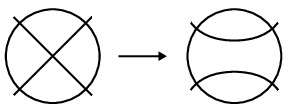}}%
\hspace*{\fill}
\caption{Two types of splits.}
\label{Fig:split}
\end{figure}

By \textit{state} is meant one of the obtained diagram with each crossing being split either by a \textit{type $ 0 $} or a \textit{type $ 1$   split}. A state can be seen as a collection of disjoint non-intersecting closed curves called \textit{circles}. For a state $ S $, we let $ |S| $ denote the number of its circles. Then, for a knot diagram $ K $ with $ m $ crossings, we define the following statistics by summing  over all states $ S $:  
\begin{equation}\label{eq:statesum}
K(x)=\sum_{S}^{} x^{|S|}=\sum_{k\geq0}^{}\sigma\left(m,k\right)x^k, 
\end{equation}
where $ \sigma(m,k) $ count the occurrence of the states with $ p $ circles, with $ \sigma(m,0)=0 $ for all $ m $. For the sake of simplicity, we call the state-sum formula \eqref{eq:statesum} the \textit{generating polynomial}. 
In fact, it is a simplified approach to the so-called Kauffman bracket polynomial \cite{GGLDSK,Kauffman}. We intentionally omit the split variables indicating the over- and under-crossing structure since the summation is calculated with respect to the shadow diagram. The generating polynomial is only intended as a tool at enumerating the state diagrams, and no attempt is made here to investigate its topological	property. Moreover, we have the following simplified rule which is then iteratively applied to all crossings in the diagram:

\begin{equation}\label{eq:rules}
\protect\includegraphics[width=.075\linewidth,valign=c]{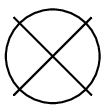}(x)=
\protect\includegraphics[width=.075\linewidth,valign=c]{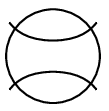}(x)+
\protect\includegraphics[width=.075\linewidth,valign=c]{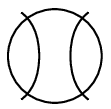}(x),
\end{equation}

In this paper, we  mainly focus on the distribution of the number $ \sigma(m,k) $ defined in \eqref{eq:statesum} in terms of generating polynomial for some particular classes of knot shadows. 

We organize the paper as follows. In \Sec{Sec:background}, we construct a recursive definition of some classes of knot shadows and define the associated closure operation. In \Sec{Sec:background}, we established the generating polynomial for the knots introduced in \Sec{Sec:gp1}. Then in \Sec{Sec:gp2}, we establish the generating polynomial for the closure of the same knots.

\section{Background}\label{Sec:background}	
Throughout this paper, unless explicitly stated otherwise,  the generic term ``knot (diagram)'' refers to a  shadow drawn on the sphere $ S^2 $.
The simplest mathematical knot is the unknot which is a closed loop with no crossings in it. We say that two knots are the same, if one can be continuously deformed to the other so long as no new crossings are introduced  and no crossings are removed. Such deformation is called a \textit{planar isotopy}. A practical  illustration would be to consider the corresponding diagram as a ``highly deformable rubber'' as suggested by Collins \cite[p.\ 12]{Adams}. To set up our framework, we introduce the following deformation which preserve as well the crossings configuration.

\begin{definition}[Denton and Doyle \cite{Denton}] \label{Dfn:0Smove}
When we have a loop on the outside edge of the  diagram, we can redraw this loop around the other side of the diagram by pulling the entire loop around across the far side of the sphere without affecting the  constraints on any of the already existing crossings (see \Figs{Fig:0Smove}). The move is called  a \textit{type 0 move on the sphere}, denoted $ 0S^2 $. 
\begin{figure}[ht]
\centering
\includegraphics[width=.8\linewidth]{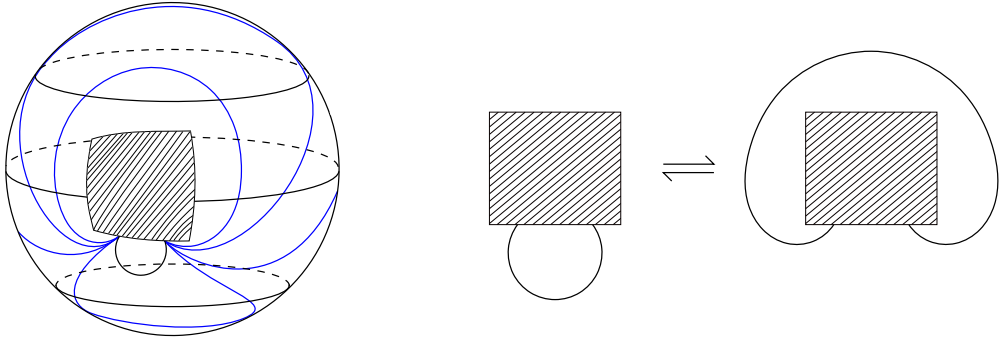}
\caption{Denton and Doyle type $ 0 $ move on the sphere.}
\label{Fig:0Smove}
\end{figure}
\end{definition}

We can then define an equivalence relation on the set of knot shadows such that two knots lie in the same equivalence class if they have the same number of crossings, and if one can be transformed to the other by a finite sequence of $ 0S^2 $ moves (modulo planar isotopy). Restricting ourself to the shadows of up to $ 3 $ crossings, we give in \Tabs{Tab:KnotShadows} all the possible combination of knot under the $ 0S^2 $ move for each given number of crossings \cite[p.\ 14]{Arnold}. We shall refer to this set of knots as \textit{elementary knots}. For our arguments, we next associate these shadow diagrams with the following operations.

\begin{table}[ht]
\centering
\begin{tabularx}{\textwidth}{|l|X|}
\toprule
$ \mathcal{S}_0 $ (\textit{unknots})& 
\protect\includegraphics[width=.05\linewidth,valign=c]{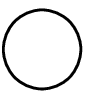}.\\
\midrule
$ \mathcal{S}_1 $ (\textit{$ 1 $-twist loops})& 
\protect\includegraphics[width=.07\linewidth,valign=c]{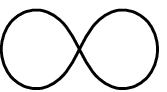}, 
\protect\includegraphics[width=.055\linewidth,valign=c]{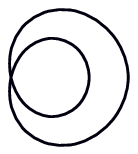}.\\
\midrule
$ \mathcal{S}_{2,1} $ (\textit{$ 1 $-links}) & 
\protect\includegraphics[width=.07\linewidth,valign=c]{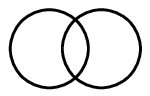}.\\
\midrule
$ \mathcal{S}_{2,2} $ (\textit{$ 2 $-twist loops})&
\protect\includegraphics[width=.10\linewidth,valign=c]{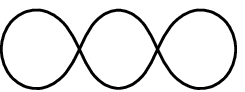}, 
\protect\includegraphics[width=.09\linewidth,valign=c]{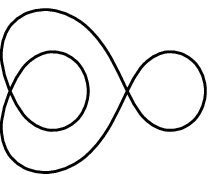}, 
\protect\includegraphics[width=.110\linewidth,valign=c]{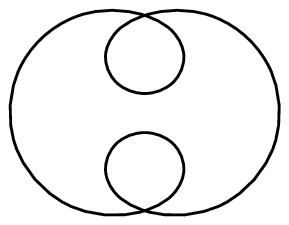}, 
\protect\includegraphics[width=.10\linewidth,valign=c]{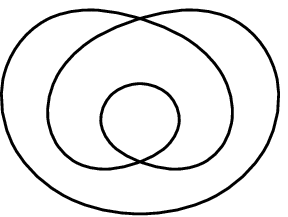}, 
\protect\includegraphics[width=.10\linewidth,valign=c]{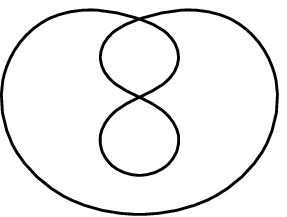}.\\
\midrule
$ \mathcal{S}_{3,1} $ (\textit{$ 3 $-twist loops})& 
\protect\includegraphics[width=.13\linewidth,valign=c]{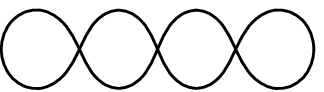}, 
\protect\includegraphics[width=.12\linewidth,valign=c]{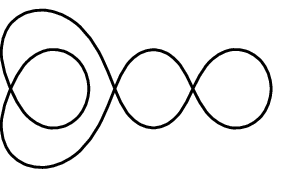}, 
\protect\includegraphics[width=.10\linewidth,valign=c]{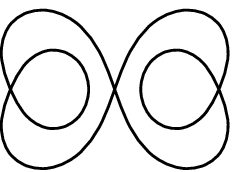}, 
\protect\includegraphics[width=.13\linewidth,valign=c]{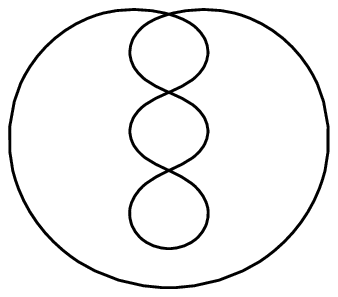}, 
\protect\includegraphics[width=.12\linewidth,valign=c]{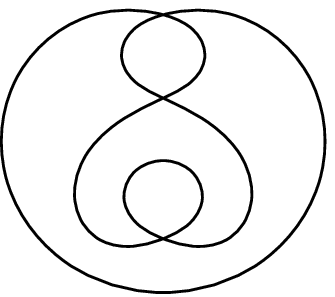}, 
\protect\includegraphics[width=.12\linewidth,valign=c]{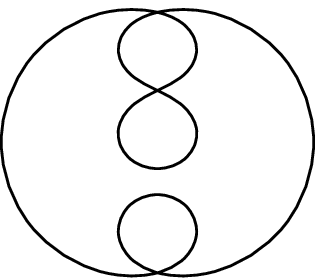},
\protect\includegraphics[width=.12\linewidth,valign=c]{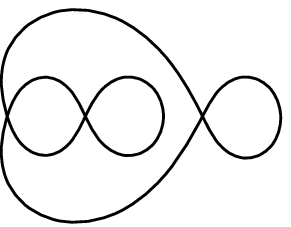}, 
\protect\includegraphics[width=.12\linewidth,valign=c]{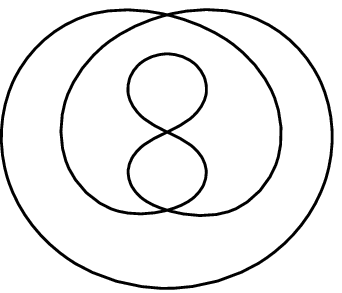}, 
\protect\includegraphics[width=.15\linewidth,valign=c]{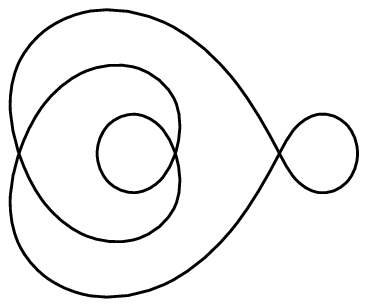}, 
\protect\includegraphics[width=.12\linewidth,valign=c]{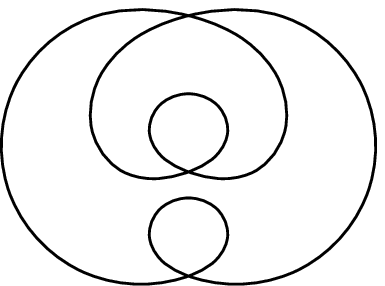}, 
\protect\includegraphics[width=.12\linewidth,valign=c]{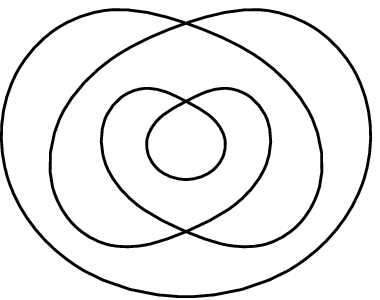}.
\\
\midrule
$ \mathcal{S}_{3,2} $ (\textit{$ 3 $-ears})&  
\protect\includegraphics[width=.100\linewidth,valign=c]{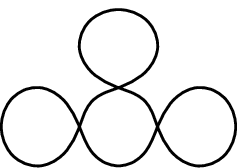}, 
\protect\includegraphics[width=.150\linewidth,valign=c]{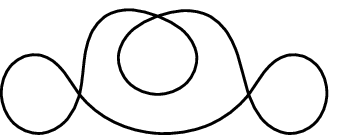}, 
\protect\includegraphics[width=.15\linewidth,valign=c]{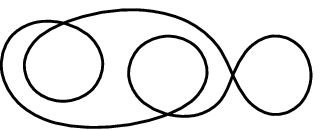}, 
\protect\includegraphics[width=.125\linewidth,valign=c]{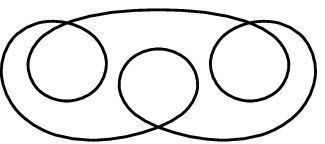}, 
\protect\includegraphics[width=.125\linewidth,valign=c]{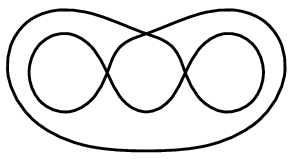}
\protect\includegraphics[width=.13\linewidth,valign=c]{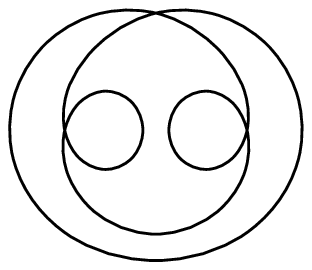}, \\
&
\protect\includegraphics[width=.125\linewidth,valign=c]{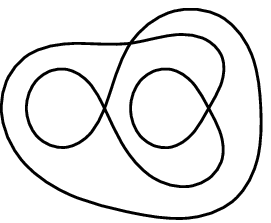}, 
\\
\midrule
$ \mathcal{S}_{3,3} $ (\textit{$ 1 $-twist links})& 
\protect\includegraphics[width=.10\linewidth,valign=c]{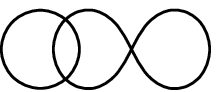},
\protect\includegraphics[width=.1\linewidth,valign=c]{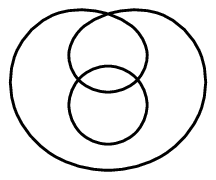},
\protect\includegraphics[width=.1\linewidth,valign=c]{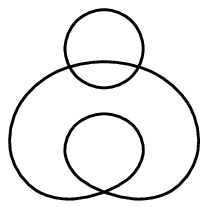},
\protect\includegraphics[width=.09\linewidth,valign=c]{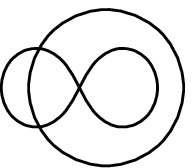},
\protect\includegraphics[width=.1\linewidth,valign=c]{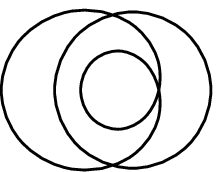}.\\
\midrule
$ \mathcal{S}_{3,4} $ (\textit{trefoils})&   
\protect\includegraphics[width=.1\linewidth,valign=c]{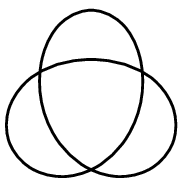},
\protect\includegraphics[width=.1\linewidth,valign=c]{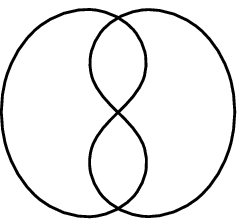}.\\
\bottomrule
\end{tabularx}
\caption{Knot shadow with at most $ 3 $ crossings}
\label{Tab:KnotShadows}
\end{table}

\begin{definition}
The \textit{connected sum} of two  knots $ K $ and $ K'  $, denoted by  $ K \#K'$, is the knot obtained by removing a small arc from each knot and then connecting the four endpoints by two new arcs in such a way that no new crossings are introduced \cite{Johnson}.

Analogously, the \textit{disconnected sum} or  \textit{disjoint union} of two  knots $ K $ and $ K'  $, denoted by  $ K \sqcup K'$, is the diagram obtained by placing the two diagrams  inside two non-intersecting domains on the sphere \cite[p.\ 15]{Manturov}. 
\end{definition}

\begin{example}
Consider the following connected sum:
\begin{figure}[H]
\centering
\includegraphics[width=.5\linewidth]{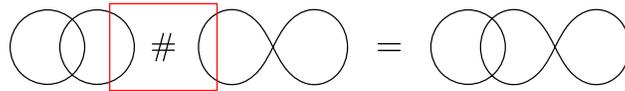}
\caption{The connected sum of a \textit{$ 1 $-link} and a \textit{$ 1 $-twist loop} gives a \textit{$ 1 $-twist link}.}
\label{Fig:ConnectedSumDef}
\end{figure}
\end{example}

We point out that the connected sum $ \# $ and the disconnected sum $ \sqcup $ are both associative and commutative \cite[p.\ 61]{Weiping}. Moreover, for any knot $ K $, we have $ K\#U=K $, where $ U $ denote the unknot.

Our framework will make extensive use of the following special notation:
\begin{notation} Let $ K $ be a knot,  and let $ n $ be a nonnegative integer.
\begin{enumerate}
\item $ K_n:=\underbrace{K\#K\#\cdots\#K}_{n\ copies} $ with $ K_0=U $. We say that the knot $ K_n $ is \textit{generated} by $ K $, and the knot $ K $ is the \textit{generator} of $ K_n $.\label{item:connexion}
\item $ K^n:=\underbrace{K\sqcup K\sqcup\cdots\sqcup K}_{n\ copies} $ with $ K^0=\varnothing $ (the empty knot).
\end{enumerate}
\end{notation}

In this paper, we establish the generating polynomials of the knots that are generated by the elementary knots. To begin with, we  pay a special attention to the following series of knots.

\begin{definition}[Twist loop]
A \textit{twist loop} is a knot obtained by twisting the unknot. We refer to a twist loop of $ n $ half twists as \textit{$ n $-twist loop} \cite{Ramaharo,RR}.  We  let $ T_n $ denote an $ n $-twist loop, with 
\begin{equation}
T_n:=\protect\includegraphics[width=0.345\linewidth,valign=c]{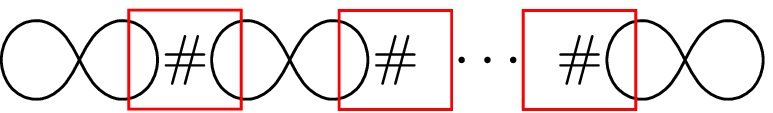}=\protect\includegraphics[width=0.2\linewidth,valign=c]{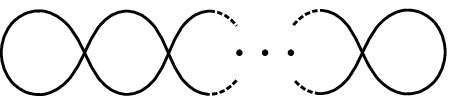}.
\end{equation}
\end{definition}

\begin{definition}[Link]
An \textit{$ n $-link} is a knot which consists of  $ n+1 $ linear interlocking circles.  We let $ L_n $ denote an $ n $-link, with 
\begin{equation}
L_n:=\protect\includegraphics[width=0.345\linewidth,valign=c]{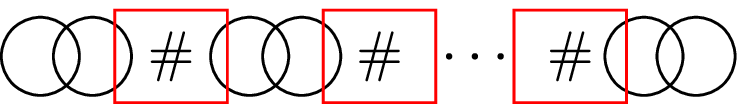}=\protect\includegraphics[width=0.16\linewidth,valign=c]{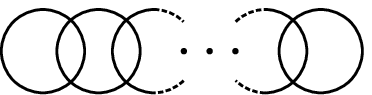}.
\end{equation}
The knot $ L_1 $ is also referred to as \textit{Hopf link}.
\end{definition}

\begin{definition}[Twist link]
We construct an \textit{$ n $-twist link}  is  by interlocking $ n $ series of $ 1 $-twist loops, starting from the unknot. We let $ W_n $ denote an $ n $-twist link, with 
\begin{equation}
W_n :=\protect\includegraphics[width=0.45\linewidth,valign=c]{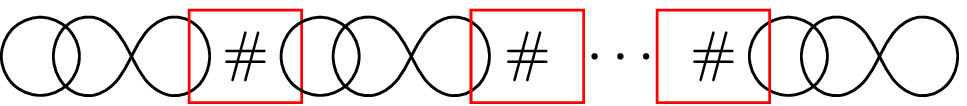}=\protect\includegraphics[width=0.28\linewidth,valign=c]{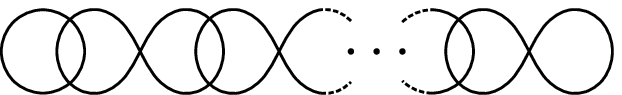},
\end{equation}  where the  generator is obtained by twisting the Hopf link.
\end{definition}

\begin{definition}[Hitch knot]
Ashley \cite[\#50, p.\ 14]{ABOK} describes the \textit{half hitch} as ``tied with one end of a rope being passed around an object and secured to its own standing part with a single hitch'', see  \Figs{Fig:singledoublehitch}.

\begin{figure}[H]
\centering
\includegraphics[width=.6\linewidth]{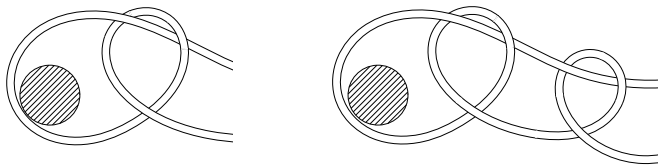}
\caption{Single half hitch and double half hitches}
\label{Fig:singledoublehitch}
\end{figure}

We define a  \textit{$ n $-hitch knot}, denoted by $ H_n $, as the shadow obtained by joining together the two loose ends of a thread of $ n $ half hitches. The corresponding connected sum is given by  
\begin{equation}
H_n:=\protect\includegraphics[width=0.345\linewidth,valign=b]{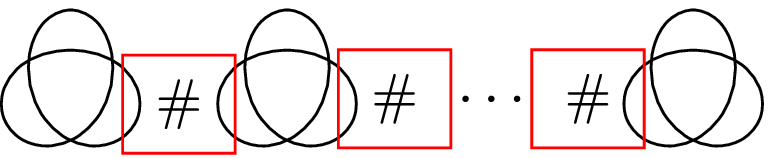}=\protect\includegraphics[width=0.21\linewidth,valign=b]{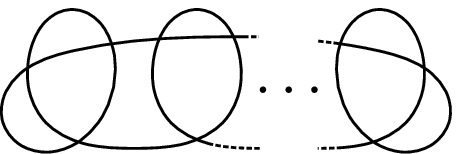}.
\end{equation} The knot $ H_1 $ is known as  \textit{trefoil}.
\end{definition}

\begin{definition}[Overhand knot] The \textit{overhand knot} is a knot obtained by making a loop in a piece of cord and pulling the end through it. For instance, we see in \Figs{fig:overhandoverhand} a single and a two series of overhand knot.
\begin{figure}[H]
\centering
\includegraphics[width=.4\linewidth]{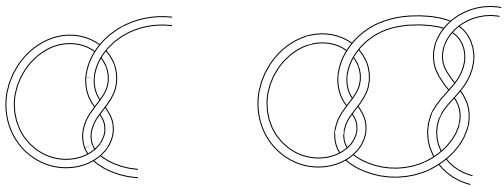}
\caption{Single overhand knot, double overhand knot (\textit{square knot}).}
\label{fig:overhandoverhand}
\end{figure}
If as previously we join together the two loose ends of a $ n $ series of overhand knot, then we call the projected shadow an \textit{$ n $-overhand knot}, and we shall refer to such knot as $ O_n $.  We have  
\begin{equation}
O_n:=\protect\includegraphics[width=0.345\linewidth,valign=c]{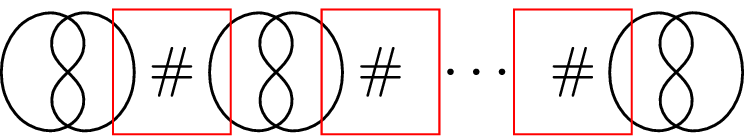}=\protect\includegraphics[width=0.18\linewidth,valign=c]{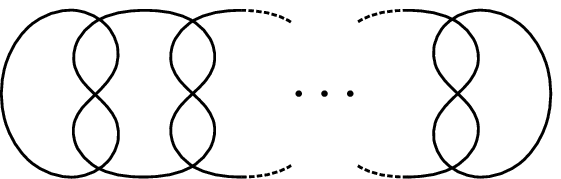}.
\end{equation} 
In the present representation, we also  refer to the knot $ O_1 $ as \textit{trefoil}.
\end{definition}

Besides, we  define as well the \textit{closure} or the \textit{closed connected sum} of a knot as the connected sum with itself as shown in \Figs{Fig:ConnectedSumClosure}. The closure of the unknot, which is a disjoint union of two closed loops, is therefore the simplest of all the closure of knots. We let $ \overline{K} $ denote a closure of the knot $ K $.
\begin{figure}[ht]
\centering
\hspace*{\fill}
\subfigure[$\overline{K}  $]{\includegraphics[width=0.175\linewidth]{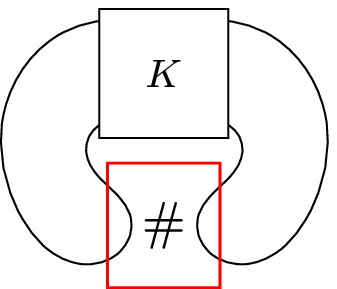}}\hfill%
\subfigure[$\overline{K_n}  $]{\includegraphics[width=0.25\linewidth]{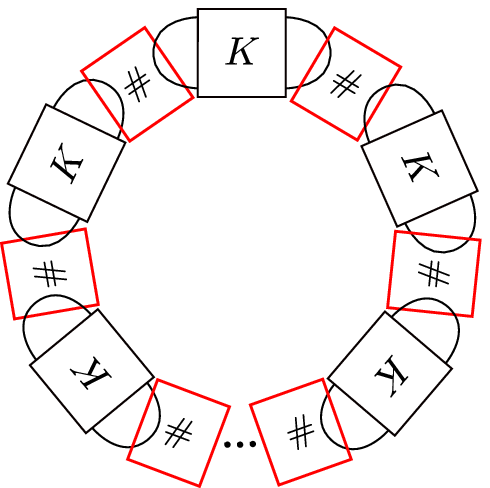}}\hfill%
\subfigure[The unknot $ U $ and its closure $ \overline{U} $]{\includegraphics[width=0.325\linewidth]{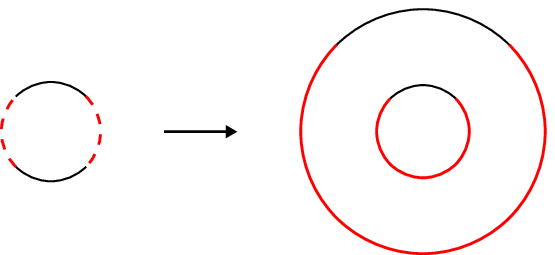}}
\hspace*{\fill}
\caption{The closed connected sum.}
\label{Fig:ConnectedSumClosure}
\end{figure}

Let us then present the \textit{$ n $-foil knot}, the \textit{$ n $-chain link}, the \textit{$ n $-twisted bracelet}, the \textit{$ n $-ringbolt hitching}, and the \textit{$ n $-sinnet of square knotting} which are respectively the closure of the $ n $-twist loop, the $ n $-link, the $ n $-twist link, the $ n $-hitch knot and the $ n $-overhand knot. In what follows, we give the formal definition of these  knots, and give the corresponding shadow diagrams.

\begin{definition}[Foil knot]
An \textit{$ n $-foil} \cite{RR} is a knot obtained by winding $ n $ times around a circle in the interior of the torus, and $ 2 $ times around its axis of rotational symmetry \cite[p.\ 107]{Adams}. 
\begin{figure}[H]
\centering
\includegraphics[width=.725\linewidth]{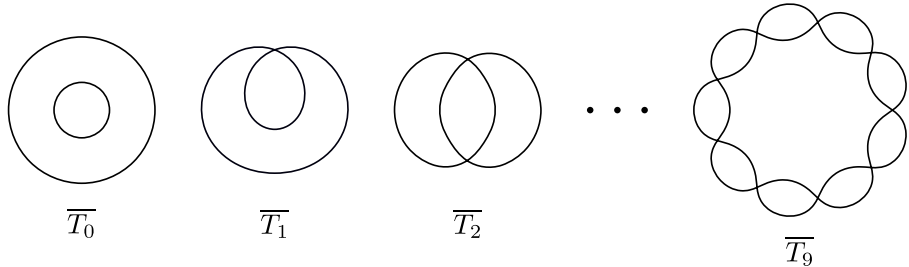}
\caption{$ n$-foil knots,  $ n=0,1,2,9$}
\label{fig:twistloopclosed}
\end{figure}
\end{definition}

\begin{definition}[Chain link]
An \textit{$ n $-chain link} consists of $ n $ unknotted circles embedded in $ S^3 $, linked together in a closed chain \cite{KPR}. 
\begin{figure}[H]
\centering
\includegraphics[width=.725\linewidth]{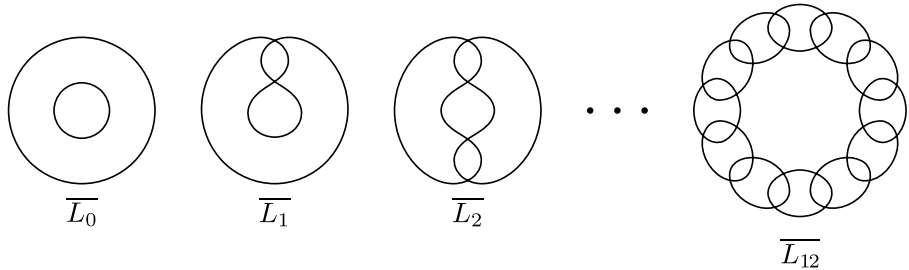}
\caption{$ n$-chain links,  $ n=0,1,2,12$.}
\end{figure}
\end{definition}

\begin{definition}[Twist bracelet]
An \textit{$ n $-twist bracelet}  (or a \textit{twisted $ n$-chain link} \cite{KPR}) consists of $ n $ twisted link intertwined together in a closed chain \cite{QM}.
\begin{figure}[H]
\centering
\includegraphics[width=.725\linewidth]{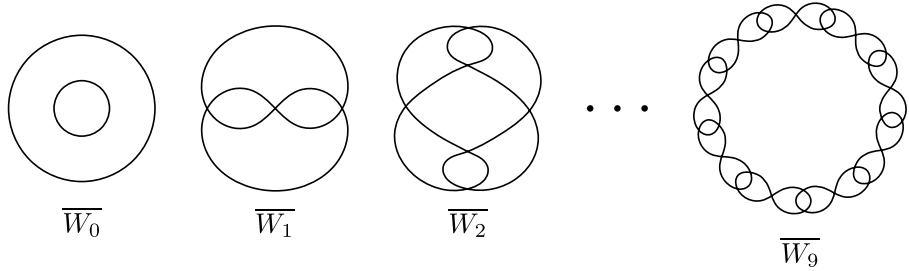}
\caption{$ n$-twist bracelets,  $ n=0,1,2,12$.}
\end{figure}
\end{definition}

\begin{definition}[Ringbolt hitching] By \textit{$ n $-ringbolt hitching}, we mean a series of $  n $ half hitches  that form a ridge around a ring or loop.
\begin{figure}[H]
\centering
\includegraphics[width=.725\linewidth]{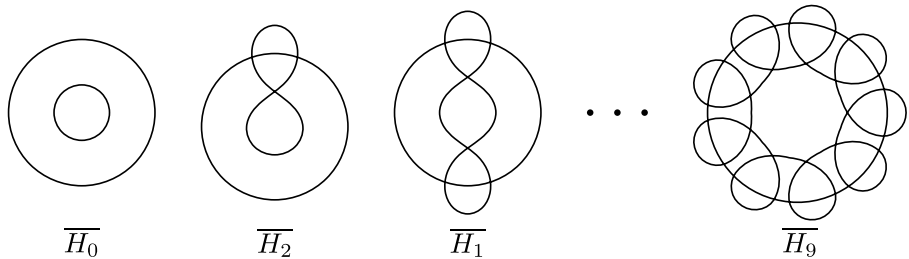}
\caption{$ n$-twisted bracelet,  $ n=0,1,2,9 $.}
\end{figure}
\end{definition}

\begin{definition}[Sinnet of square knotting] Ashley \cite[\#2906, p.\ 471]{ABOK} defines a \textit{chain sinnet}  as a knot which are made of one or more strands that are formed into successive loops, which are tucked though each other. Here, we borrow the term \textit{$ n $-sinnet of square knotting} to describe  a closed chain of $ n $ overhand knot, see \Figs{Fig:SinnetOfSquare}.
\begin{figure}[H]
\centering
\includegraphics[width=.725\linewidth]{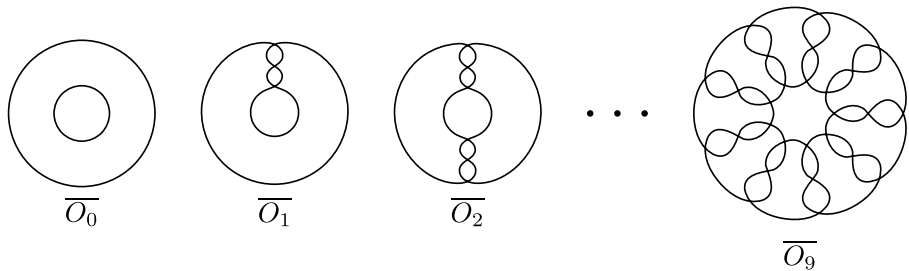}
\caption{$ n$-sinnet of square knotting,  $ n=0,1,2,9 $.}
\label{Fig:SinnetOfSquare}
\end{figure}
\end{definition}

\begin{remark}
The  pairs of knots shown in \Figs{Fig:EquivalentKnot} are equivalent under the $ 0S^2 $ move. The  $ 0S^2 $ move does not remove nor create a crossing, therefore a complete split leads to the same state diagram. We then expect that knots belonging to same equivalence class  have equals generating polynomials.
\begin{figure}[ht]
\centering
\hspace*{\fill}
\subfigure[$  \overline{T_1}= T_1  $]{\includegraphics[width=0.25\linewidth]{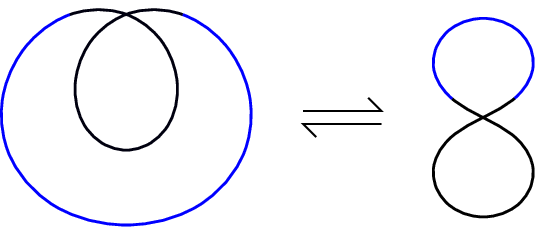}}\hfill%
\subfigure[$ \overline{L_1}= T_2$]{\includegraphics[width=0.25\linewidth]{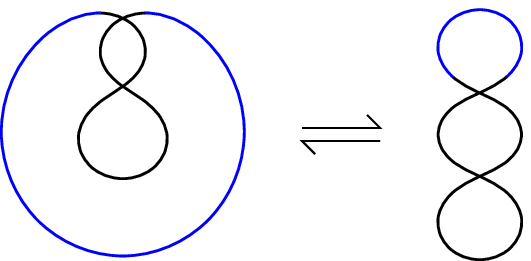}}\hfill%
\subfigure[$ \overline{O_1}= T_3 $]{\includegraphics[width=0.25\linewidth]{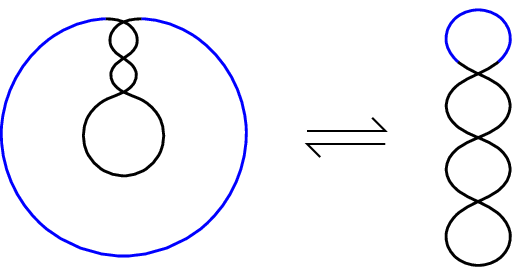}}
\hspace*{\fill}
\\
\hspace*{\fill}
\subfigure[$ \overline{H_1}= W_1 $]{\includegraphics[width=0.25\linewidth]{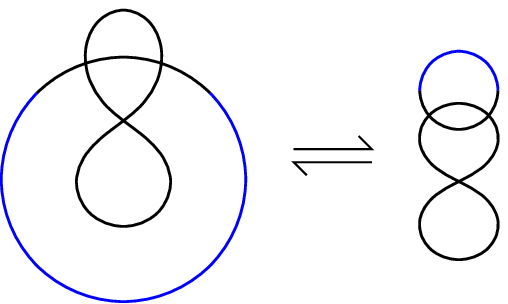}}\hfill%
\subfigure[$ \overline{L_2}= \overline{T_4}$]{\includegraphics[width=0.3\linewidth]{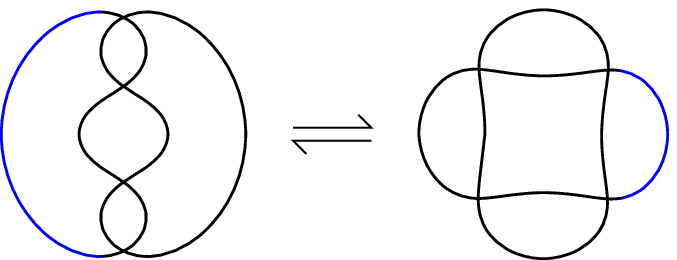}}
\hspace*{\fill}
\\
\hspace*{\fill}
\subfigure[$ \overline{O_2}= \overline{T_6}$]{\includegraphics[width=0.3\linewidth]{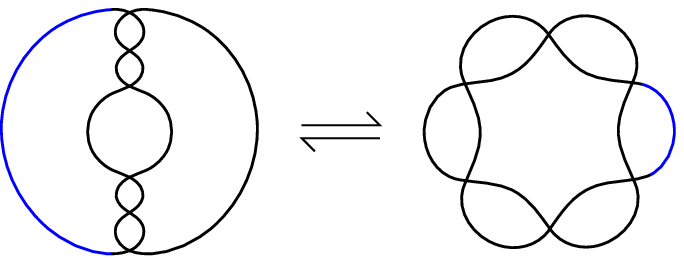}}\hfill%
\subfigure[$ \overline{W_1}= O_1= H_1$]{\includegraphics[width=0.35\linewidth]{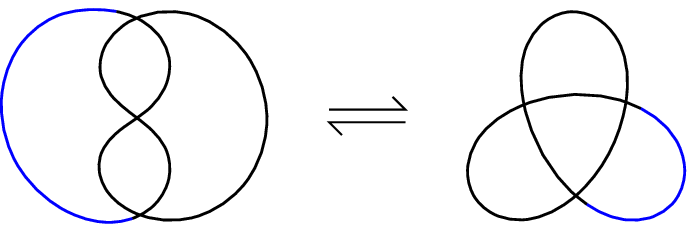}}
\hspace*{\fill}
\caption{Equivalent knots under the $ 0S^2 $ move.}
\label{Fig:EquivalentKnot}
\end{figure}
\end{remark}

\section{The generating polynomial for the connected sum}\label{Sec:gp1}
The present section is devoted to computing the generating polynomials of the previously introduced knots. 
\subsection{Preliminaries}
Regarding the knot operations $ \# $ and $ \sqcup $, we have the following immediate results:
\begin{align*}
U(x)=x\ \mbox{and}\ U^n(x)=x^n.
\end{align*}
A first extension is as follows.
\begin{proposition}\label{Prop:sqcupU}
For an arbitrary knot $ K $ and the unknot $U$, the following property holds
\begin{equation}
(K\sqcup U)(x)=xK(x).
\end{equation}
\end{proposition}

\begin{proof}
The initial diagram is accompanied with the unknot and so are each resulting diagram after a series of splits. So if $ K(x)=\sum_{S}^{} x^{|S|} $, then $ (K\sqcup U)(x)=\sum_{S}^{} x^{|S|+1}=x\sum_{S}^{} x^{|S|} $.
\end{proof}

\begin{corollary}\label{Cor:sqcupUn}
Let $ n $ be a nonnegative integer, and let $ K $ be an arbitrary knot. Then
\begin{equation*}
\left(K\sqcup U^n\right)(x)=x^nK(x).
\end{equation*}
\end{corollary}
We can generalize \Prop{Prop:sqcupU} as follows
\begin{proposition}
For two arbitrary knots $ K $ and $ K' $, the following equality holds
\begin{equation}\label{eq:sqcup}
(K\sqcup K')(x)=K(x).K'(x).
\end{equation}
\end{proposition}

\begin{proof}
Let us first compute the states of $ K $. Each of these states is accompanied with the diagram of $ K' $. Hence we have
\begin{equation*}
(K\sqcup K')(x)=\sum_{i\geq1}^{}\left(U^{k_i}\sqcup K'\right)(x),
\end{equation*}
where $ i $ runs over all the states of  $ K' $, and $ k_i $ is the corresponding number of circles. Finally by  \Cor{Cor:sqcupUn} we obtain
\begin{equation*}
(K\sqcup K')(x)=\sum_{i\geq1}^{}x^{k_i}K'(x),
\end{equation*}
and we conclude by noting that $ K(x)=\sum_{i\geq1}^{}x^{k_i}$.
\end{proof}

\begin{proposition}\label{Prop:connectedsum}
For two arbitrary knots $ K $ and $ K' $, the following equality holds
\begin{equation}\label{eq:connectedsum}
\left(K\# K'\right)(x)=x^{-1}K(x).K'(x).
\end{equation}
\end{proposition}

\begin{proof}
If we first compute the states of the knot $ K' $, then for each of these states, there exists exactly one circle which is connected to $ K' $. In terms of polynomial, it means
\begin{equation*}
(K\sqcup K')(x)=\sum_{i\geq1}^{}x^{-1}\left(U^{k_i}\sqcup K'\right)(x),
\end{equation*}
where, as previously, $ i $ runs over all the states of  $ K $, and $ k_i $ is the corresponding number of circles. The result immediately follows from formula \eqref{eq:sqcup}.
\end{proof}

\begin{corollary}
Let $ K $, $ K' $ and $ K'' $ be three knots. Then
\begin{align}
(K\#K')(x)&=(K'\#K)(x),\label{eq:commut}\\ 
\big(\left(K\#K'\right)\#K''\big)(x)&=\big(K\#\left(K'\#K''\right)\big)(x).\label{eq:assoc}
\end{align}
\end{corollary}

We will make extensive use of the following particular case
\begin{corollary}\label{cor:genpoly}
For an arbitrary knot $ K $ and a nonnegative integer $ n $, we have
\begin{equation}\label{eq:generatorp}
K_n(x) = x\big(x^{-1}K(x)\big)^n.
\end{equation}
\end{corollary}

\begin{remark}\label{Rem:generator}
Formula \eqref{eq:generatorp} suggests that in order to compute the generating polynomial for the knot $ K_n $, we simply have to compute that of the generator $ K $. This  formula also means that the generating polynomials associated with knots generated by elementary knots that lie in the same equivalent class are exactly the same, regardless of the choice of the arcs at which the connected sum is performed.
\end{remark}

\begin{theorem}\label{Thm:gf}
The generating function for the sequence $ \big\{K_n(x)\big\}_{n\geq0} $ with respect to the occurrence of the generator $ K $ (marked by $ y $) and the state diagrams (marked by $ x $) is
\begin{equation*}
K(x;y):=\dfrac{x^2}{x-yK(x)}.
\end{equation*}
\end{theorem}

\begin{proof}
We write $ K(x;y):=\sum_{n\geq 0}^{}K_n(x)y^n $, and the result follows from formula \eqref{eq:generatorp}.
\end{proof}

The results in subsection  \ref{subsec:twist-loop} -- \ref{subsec:overhand-knot} are all then immediate application of \Cor{cor:genpoly} and \Thm{Thm:gf}. For each of the concerned knots, we give the generating polynomial and the associated generating function. If available, we also give the A-records and the short definition from the \textit{On-Line Encyclopedia of Integer Sequences} (OEIS) \cite{Sloane}.

\subsection{Twist loop}\label{subsec:twist-loop}
Let $ T_n(x) :=\sum_{k\geq 0}^{}t\left(n,k\right)x^k$ denote the generating polynomial for the $ n $-twist loop. 

\begin{theorem}
The generating polynomial for the $ n $-twist loop is given by the recurrence relation
\begin{equation}\label{eq:Tnx1}
T_n(x)=(x+1)T_{n-1}(x),
\end{equation}
and is expressed by the closed form formula
\begin{equation}\label{eq:Tnx2}
T_n(x)=x(x+1)^n.
\end{equation}
\end{theorem}

\begin{proof}
See \Figs{Fig:1TwistLoop}, then apply \Cor{cor:genpoly}.
\begin{figure}[ht]
\centering
\includegraphics[width=.25\linewidth]{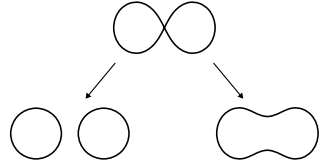}
\caption{The states of the $ 1 $-twist loop, $ 	T_1(x)=x^2+x $.}
\label{Fig:1TwistLoop}
\end{figure}
\end{proof}

\begin{corollary}
The generating function for the  sequence $ \big\{T_n(x)\big\}_{n\geq0} $ is given by
\begin{equation}
T(x;y):=\dfrac{x}{1-y(x+1)}.
\end{equation}
\end{corollary}

Combining expressions \eqref{eq:Tnx1} and \eqref{eq:Tnx2}, we obtain the following recurrence relation:
\begin{equation}\label{eq:tnk}
\begin{cases}
t(n,0)=0,\ t(n,1)=1, & n\geq 0;\\
t(n,k)=t(n-1,k-1)+t(n-1,k),	&k\geq 1,\ n\geq 0.
\end{cases}
\end{equation}
The values for the array $ \left(t(n,k)\right)_{n\geq0,\  k\geq 0} $ are given in \Tabs{Tab:twistloop} for small value of $ n $ and $ k $, with $ n\geq k-1 $.
The result is a horizontal-shifted Pascal's triangle \cite[\seqnum{A007318}]{Sloane}
\begin{table}[ht]
\centering
$\begin{array}{c|rrrrrrrrrrrrrr}
n\ \backslash\ k	&0	&1	&2	&3	&4	&5	&6	&7	&8	&9	&10	&11\\
\midrule
0	&0	&1	&	&	&	&	&	&	&	&	&	&	\\
1	&0	&1	&1	&	&	&	&	&	&	&	&	&	\\
2	&0	&1	&2	&1	&	&	&	&	&	&	&	&	\\
3	&0	&1	&3	&3	&1	&	&	&	&	&	&	&\\
4	&0	&1	&4	&6	&4	&1	&	&	&	&	&	&\\
5	&0	&1	&5	&10	&10	&5	&1	&	&	&	&	&\\
6	&0	&1	&6	&15	&20	&15	&6	&1	&	&	&	&	\\
7	&0	&1	&7	&21	&35	&35	&21	&7	&1	&	&	&	\\
8	&0	&1	&8	&28	&56	&70	&56	&28	&8	&1	&	&	\\
9	&0	&1	&9	&36	&84	&126	&126	&84	&36	&9	&1	&	\\
10	&0	&1	&10	&45	&120	&210	&252	&210	&120	&45	&10	&1	\\
\end{array}$
\caption{Values of $ t(n,k) $ for  $ 0\leq n\leq 10 $ and $ 0\leq k\leq 11 $.}
\label{Tab:twistloop}
\end{table}

\begin{remark}\label{Rem:T2T3}
Let $ K\in\mathcal{S}_{2,2}\cup\mathcal{S}_{3,1}\cup\mathcal{S}_{3,2}  $. Then by \Rem{Rem:generator}, we have
\begin{equation}
K_n(x)=\begin{cases}
\left(T_2\right)_n(x) ,&\textit{if $ K\in\mathcal{S}_{2,2} $};\\
\left(T_3\right)_n(x),&\textit{if $ K\in\mathcal{S}_{3,1}\cup\mathcal{S}_{3,2} $},\\
\end{cases}
\end{equation}
with 
\begin{align*}
\left(T_{2}\right)_n:&=\protect\includegraphics[width=0.55\linewidth,valign=c]{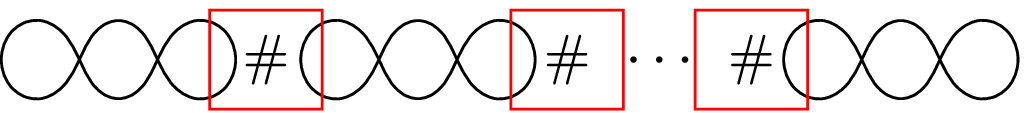};\\
\left(T_{3}\right)_n:&=\protect\includegraphics[width=0.7\linewidth,valign=c]{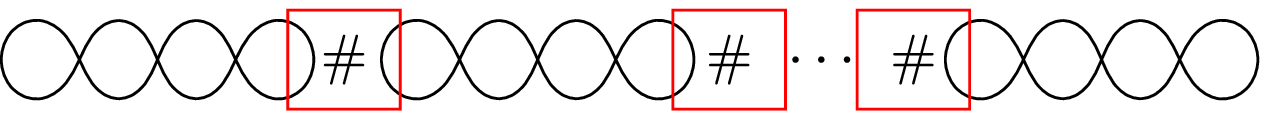}.
\end{align*}
Here, it is immediate that  for all nonnegative integer $ i $,  $ \left(T_i\right)_n=T_{in} $.  Let us  then introduce the following calculations to complete our results.
\paragraph{{\boldmath$2n $-twist loop}:}
let  $T_{2n}(x):=\sum_{k\geq 0}^{}t_2(n,k)x^k $.
\begin{enumerate}
\item Generating polynomial: \begin{equation}
T_{2n}(x)=x\left(x+2x+1\right)^{n}.
\end{equation}
\item Generating function: \begin{equation}
T_2(x;y)=\dfrac{x}{1-y(x^2+2x+1)}.
\end{equation}
\item Distribution of $ t_2(n,k) $: see \Tabs{Tab:T2nk}.
\begin{equation}
\begin{cases}
t_2(n,0)=0,\ t_2(n,1)=1,\ t_2(n,2)=2n, & n\geq 0;\\
t_2(n,k)=t_2(n-1,k-2)+2t_2(n-1,k-1)+t_2(n-1,k),	&k\geq 2,\ n\geq 0.
\end{cases}
\end{equation}
When $ k\geq 1 $, the triangle $ \left(t_2(n,k)\right)_{n\geq0}$ is  a horizontal-shifted even-numbered rows of Pascal's triangle \cite[\seqnum{A034870}]{Sloane}.
\begin{table}[ht]
\centering
$\begin{array}{c|rrrrrrrrrrrrrrrr}
n\ \backslash\ k		 &0		 &1		 &2		 &3		 &4		 &5		 &6		 &7		 &8		 &9		 &10		 &11	 &12	 &13	 &14	 &15\\
\midrule
0	 &0	 &1	 &	 &	 &	 &	 &	 &	 &	 &	 &	 &	 &	 &	 &	 &\\
1	 &0	 &1	 &2	 &1	 &	 &	 &	 &	 &	 &	 &	 &	 &	 &	 &	 &\\
2	 &0	 &1	 &4	 &6	 &4	 &1	 &	 &	 &	 &	 &	 &	 &	 &	 &	 &\\
3	 &0	 &1	 &6	 &15	 &20	 &15	 &6	 &1	 &	 &	 &	 &	 &	 &	 &	 &\\
4	 &0	 &1	 &8	 &28	 &56	 &70	 &56	 &28	 &8	 &1	 &	 &	 &	 &	 &	 &\\
5	 &0	 &1	 &10	 &45	 &120	 &210	 &252	 &210	 &120	 &45	 &10	 &1	 &	 &	 &	 &\\
6	 &0	 &1	 &12	 &66	 &220	 &495	 &792	 &924	 &792	 &495	 &220	 &66	 &12	 &1	 &	 &\\
7	 &0	 &1	 &14	 &91	 &364	 &1001	 &2002	 &3003	 &3432	 &3003	 &2002	 &1001	 &364	 &91	 &14	 &1
\end{array}$
\caption{Values of $ t_2(n,k) $ for $ 0\leq n\leq 7 $ and $ 0\leq k\leq 15 $.}
\label{Tab:T2nk}
\end{table}
\end{enumerate}

\paragraph{\boldmath$3n $-twist loop:} let $T_{3n}(x):=\sum_{k\geq 0}^{}t_3(n,k)x^k $. 
\begin{enumerate}
\item Generating polynomial:  \begin{equation}
T_{3n}(x)=x\left(x^3+3x^2+3x+1\right)^{n}.
\end{equation}
\item Generating function: \begin{equation}
T_3(x;y):=\dfrac{x}{1-y(x^3+3x^2+3x+1)}.
\end{equation}
\item Distribution of $ t_3(n,k) $: see \Tabs{Tab:T3nk} (also, refer back to \Tabs{Tab:twistloop}).
\begin{equation}
\begin{cases}
t_3(n,0)=0,\ t_3(n,1)=1,\ t_3(n,2)=3n,\ t_3(n,3)=\dfrac{3n(3n-1)}{2}, & n\geq 0;\\
t_3(n,k)=t_3(n-1,k-3)+3t_3(n-1,k-2)\\
\hphantom{t_3(n,k)=}+3t_3(n-1,k-1)+t_3(n-1,k),	&k\geq 3, n\geq 0.
\end{cases}
\end{equation}
The triangle $ \left(t_3(n,k)\right)_{k\geq 1,\ n\geq0}$ is  given by  $ \binom{3n}{k-1} $ \cite[$ \seqnum{A007318}(3n,k-1) $]{Sloane}.
\begin{table}[ht]
\centering
\resizebox{\linewidth}{!}{%
$\begin{array}{c|rrrrrrrrrrrrrrrrr}
n\ \backslash\ k	&0	&1	&2	&3	&4	&5	&6	&7	&8	&9	&10	&11	&12&13&14&15&16\\
\midrule
0	&0	&1	&	&	&	&	&	&	&	&	&	&	&&&&&\\
1	&0	&1	&3	&3	&1	&	&	&	&	&	&	&	&&&&&\\
2	&0	&1	&6	&15	&20	&15	&6	&1	&	&	&	&	&&&&&\\
3	&0	&1	&9	&36	&84	&126	&126	&84	&36	&9	&1	&	&&&&&\\
4	&0	&1	&12	&66	&220	&495	&792	&924	&792	&495	&220	&66	&12&1&&&\\
5&0&1&15&105&455&1365&3003&5005&6435&6435&5005&3003&1365&455&105&15&1
\end{array}$
}
\caption{Values of $ t_3(n,k) $ for  $ 0\leq n\leq 5 $ and $ 0\leq k\leq 16 $.}
\label{Tab:T3nk}
\end{table}
\end{enumerate}
\end{remark}
\subsection{Link}\label{subsec:links}
Let $ L_n(x) :=\sum_{k\geq 0}^{}\ell(n,k)x^k$ denote the generating polynomial for the $ n $-link. 

\begin{theorem}
The generating polynomial for the $ n $-link is given by the recurrence relation
\begin{equation}
L_n (x)=\left(2x+2\right)L_{n -1}(x),
\end{equation}
and is expressed by the closed form formula
\begin{equation}\label{Lnx2}
L_n(x)=x(2x+2)^n.
\end{equation}
\end{theorem}

\begin{proof}
\begin{figure}[ht]
\centering
\includegraphics[width=.4\linewidth]{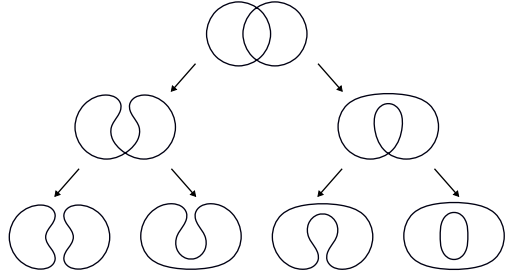}
\caption{The states of the $ 1 $-link, $ L_1(x)=2x^2+2x $.}
\label{Fig:1LinkStates}
\end{figure}
See \Figs{Fig:1LinkStates}.
\end{proof}

\begin{corollary}
The generating function for the  sequence $ \big\{L_n(x)\big\}_{n\geq0} $ is given by
\begin{equation}
L(x;y):=\dfrac{x}{1-y(2x+2)}.
\end{equation}
\end{corollary}

The expression of the polynomial $ L_n(x) $ suggests that 
\begin{equation*}
\ell(n,k)=2^nt(n,k)=2^n\binom{n}{k-1},\ k\geq 1, \ n\geq 0\ \mbox{\cite[\seqnum{A038208}]{Sloane}}.
\end{equation*}
Therefore, the corresponding polynomial coefficients satisfy the following recurrence relation:
\begin{equation}\label{eq:lnk}
\begin{cases}
\ell(n,0)=0,\ \ell(n,1)=2^n, 				& n\geq 0;\\
\ell(n,k)=2^n\big(t(n-1,k)+t(n-1,k-1)\big)\\
\hphantom{\ell(n,k)}=\ell(n-1,k)+\ell(n-1,k-1),	&k\geq 1,\ n\geq 0.
\end{cases}
\end{equation}
Therefore we have \Tabs{tab:link} giving the numbers $ \ell(n,k) $, for $ 0\leq n\leq 9 $ and $ 0\leq k\leq 10 $.
\begin{table}[ht]

\centering
$\begin{array}{c|rrrrrrrrrrr}
n\ \backslash\ k	&0	&1	&2	&3	&4	&5	&6	&7	&8	&9	&10\\
\midrule
0	&0	&1		&		&		&		&		&		&		&		&		&\\
1	&0	&2		&2		&		&		&		&		&		&		&		&\\
2	&0	&4		&8		&4		&		&		&		&		&		&		&\\
3	&0	&8		&24		&24		&8		&		&		&		&		&		&\\
4	&0	&16		&64		&96		&64		&16		&		&		&		&		&\\
5	&0	&32		&160	&320	&320	&160	&32		&		&		&		&\\
6	&0	&64		&384	&960	&1280	&960	&384	&64		&		&		&\\
7	&0	&128	&896	&2688	&4480	&4480	&2688	&896	&128	&		&\\
8	&0	&256	&2048	&7168	&14336	&17920	&14336	&7168	&2048	&256	&\\
9	&0	&512	&4608	&18432	&43008	&64512	&64512	&43008	&18432	&4608	&512\\
\end{array}$

\caption{Values of $ \ell(n,k) $ for $ 0\leq n\leq 9 $ and $ 0\leq k\leq 10 $.}
\label{tab:link}
\end{table}

\begin{remark}\label{rem:linktwist}
We have  $ L_n(x)=2^nT_n(x) $. We interpret the factor $ 2^n $ as follows: at each pair of crossings making up a ``link part'', if we split one  crossing in either type $ 0 $ or $ 1 $, then the resulting knot has equals generating with a twist loop. For example, we take $ n=7 $. Letting the splits series be identified by a $ 0$-$1 $'s string, we obtain some of the possible splits:
\begin{itemize}
\item $0000000 $: \protect\includegraphics[width=0.875\linewidth,valign=c]{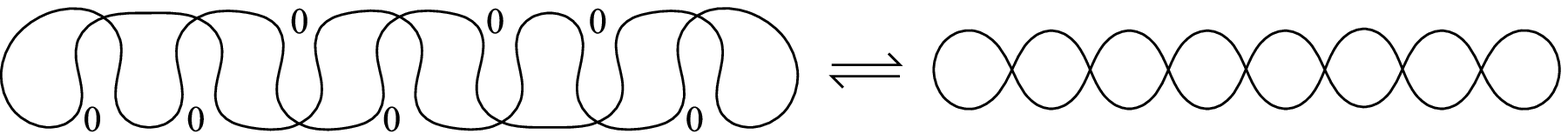};
\item $1111111 $: \protect\includegraphics[width=0.875\linewidth,valign=c]{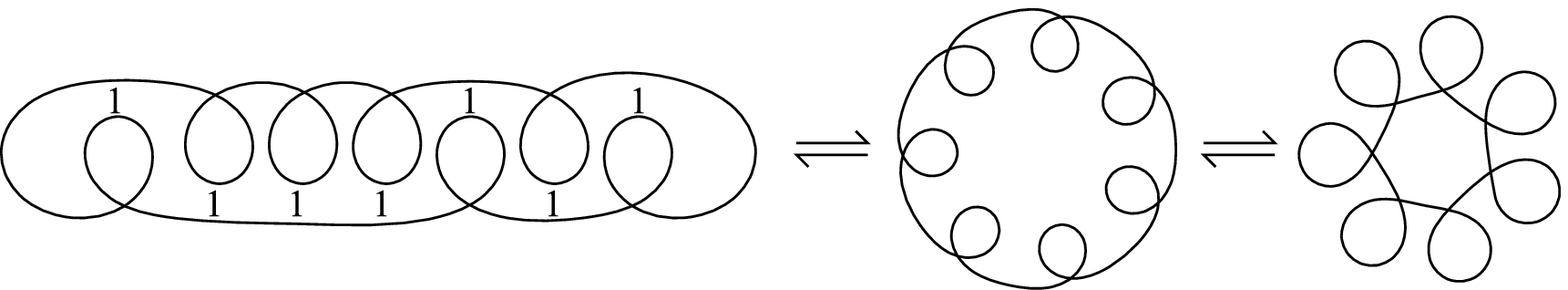};
\item $1000110 $: \protect\includegraphics[width=0.875\linewidth,valign=c]{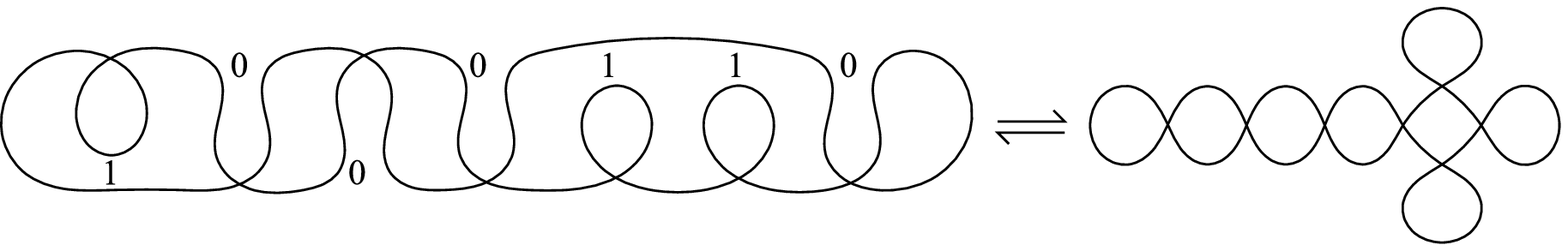}.
\end{itemize}
\end{remark}
\subsection{Twist link}\label{subsec:twisted-link}
Let $ W_n(x) :=\sum_{k\geq 0}^{}w(n,k)x^k$ denote the generating polynomial for the $ n $-twist link.

\begin{theorem}
The generating polynomial for the $ n $-twist link is given by the recurrence relation
\begin{equation}\label{eq:Wnx1}
W_n(x)=\left(2x^2+4x+2\right)W_{n-1}(x),
\end{equation}
and is expressed by the closed form formula
\begin{equation}\label{eq:Wnx2}
W_{n}(x)=x\left(2x^2+4x+2\right)^n.
\end{equation}
\end{theorem}

\begin{proof}
Note that $ W_1=L_1\#T_1 $. Then by \Prop{Prop:connectedsum} we have
\begin{align*}
W_1(x)&=x^{-1}L_1(x)T_1(x)\\
&=x^{-1}\left(2x^2+2x\right)\left(x^2+x\right)\\
&=2x^3+4x^2+2x.
\end{align*}
We conclude by \Cor{cor:genpoly}.
\end{proof}

\begin{corollary}
The generating function for the  sequence $ \big\{W_n(x)\big\}_{n\geq0} $ is given by
\begin{equation}
W(x;y):=\dfrac{x}{1-y(2x^2+4x+2)}.
\end{equation}
\end{corollary}

The recurrence relation  \eqref{eq:Wnx1} along with the closed formula \eqref{eq:Wnx2} allow us to write:
\begin{equation}\label{eq:wnk}
\begin{cases}
w(n,0)=0,\ w(n,1)=2^n,\ w(n,2)=n2^{n+1}, & n\geq 0;\\
w(n,k)=2w(n-1,k-2)+4w(n-1,k-1)+2w(n-1,k),	&k\geq 2,\ n\geq 0.
\end{cases}
\end{equation}
For $ k\geq 1 $, we obtain $w(n,k)=2^n\binom{2n}{k-1}$  \cite[\seqnum{A139548}]{Sloane}. Hence we have \Tabs{Tab:Twistlink}.

\begin{table}[ht]
\centering
\resizebox{\linewidth}{!}{%
$\begin{array}{c|rrrrrrrrrrrrrr}
n\ \backslash\ k		 &0		 &1		 &2		 &3		 &4		 &5		 &6		 &7		 &8		 &9		 &10		 &11	 &12	 &13\\
\midrule
0	 &0	 &1	 &	 &	 &	 &	 &	 &	 &	 &	 &	 &	 &	 &\\
1	 &0	 &2	 &4	 &2	 &	 &	 &	 &	 &	 &	 &	 &	 &	 &\\
2	 &0	 &4	 &16	 &24	 &16	 &4	 &	 &	 &	 &	 &	 &	 &	 &\\
3	 &0	 &8	 &48	 &120	 &160	 &120	 &48	 &8	 &	 &	 &	 &	 &	 &\\
4	 &0	 &16	 &128	 &448	 &896	 &1120	 &896	 &448	 &128	 &16	 &	 &	 &	 &\\
5	 &0	 &32	 &320	 &1440	 &3840	 &6720	 &8064	 &6720	 &3840	 &1440	 &320	 &32	 &	 &\\
6	 &0	 &64	 &768	 &4224	 &14080	 &31680	 &50688	 &59136	 &50688	 &31680	 &14080	 &4224	 &768	 &64\\
\end{array}$
}
\caption{Values of $ w(n,k) $ for $ 0\leq n\leq 6 $ and $ 0\leq k\leq 13 $.}
\label{Tab:Twistlink}
\end{table}

\subsection{Hitch knot}\label{subsec:hitch-knot}
Let $ H_n(x) :=\sum_{k\geq 0}^{}h(n,k)x^k$ denote the generating polynomial for the $ n $-hitch knot.

\begin{theorem}
The generating polynomial for the $ n $-hitch knot is given by the recurrence relation
\begin{equation}\label{eq:Hnx1}
H_n(x)=\left(x^2+4x+3\right)H_{n-1}(x),
\end{equation}
and is expressed by the closed form formula
\begin{equation}\label{eq:Hnx2}
H_n(x)=x\left(x^2+4x+3\right)^n.
\end{equation}
\end{theorem}
\begin{proof}
See \Figs{Fig:TrefoilStates}.
\begin{figure}[ht]
\centering
\includegraphics[width=.7\linewidth]{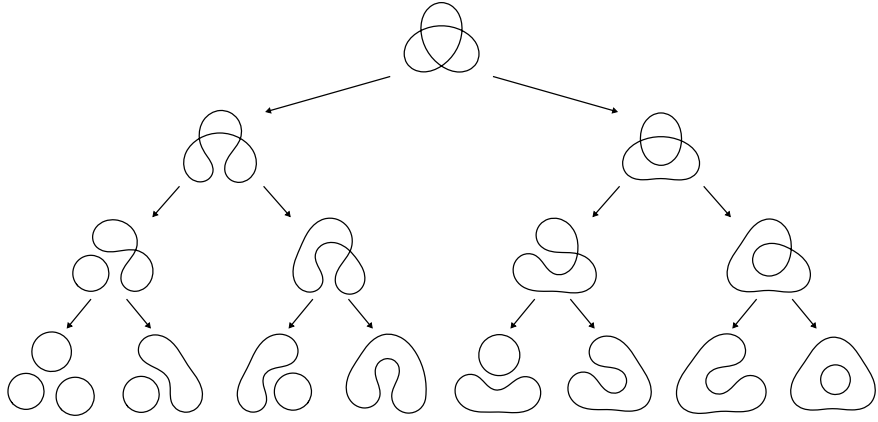}
\caption{The states of the $ 1 $-hitch, $ H_1(x)=x^3+4x^2+3x $.}
\label{Fig:TrefoilStates}
\end{figure}
\end{proof}
\begin{corollary}
The generating function for the  sequence $ \big\{H_n(x)\big\}_{n\geq0} $ is given by
\begin{equation*}
H(x;y):=\dfrac{x}{1-y\left(x^2+4x+3\right)}.
\end{equation*}
\end{corollary}

By \eqref{eq:Hnx1} and \eqref{eq:Hnx2}, we deduce the following recurrence relation:
\begin{equation}\label{eq:hnk}
\begin{cases}
h(n,0)=0,\ h(n,1)=3^n,\ h(n,2)=4(n-1)3^{n-2}, & n\geq 0;\\
h(n,k)=h(n-1,k-2)+4h(n-1,k-1)+3h(n-1,k),	&k\geq 2,\ n\geq 0.
\end{cases}
\end{equation}
The recurrence relation \eqref{eq:hnk} is used to generate the entries in  \Tabs{Tab:hitch} \cite[\seqnum{A299989}]{Sloane}.
Particularly, we recognize the following sequences:
\begin{itemize}
\item $ h(n,1)= 3^n$, the powers of $ 3 $ \cite[\seqnum{A000244}]{Sloane}; 	
\item $ h(n,2)=4n3^{n-1} $, the sum of the lengths of the drops in all ternary words of length $ n+1 $ on $ \{0,1,2\}  $ \cite[\seqnum{A120908}]{Sloane};  	
\item $ h(n,n+1)=\sum_{k=0}^n {\binom{n}{k}}^2 3^k $  \cite[\seqnum{A069835}]{Sloane};
\item $ h(n,2n+1)=1$, the all $ 1 $'s sequence \cite[\seqnum{A000012}]{Sloane};
\item $ h(n,2n-1)=n(8n-5)$ \cite[\seqnum{A139272}]{Sloane};
\item $ h(n,2n)=4n$,  the multiples of 4 \cite[\seqnum{A008586}]{Sloane}.
\end{itemize}

\begin{table}[ht]
\centering
\resizebox{\linewidth}{!}{%
$\begin{array}{c|rrrrrrrrrrrrrr}
n\ \backslash\ k		 &0		 &1		 &2		 &3		 &4		 &5		 &6		 &7		 &8		 &9		 &10		 &11	 &12	 &13\\
\midrule
0	 &0	 &1	 &	 &	 &	 &	 &	 &	 &	 &	 &	 &	 &	 &\\
1	 &0	 &3	 &4	 &1	 &	 &	 &	 &	 &	 &	 &	 &	 &	 &\\
2	 &0	 &9	 &24	 &22	 &8	 &1	 &	 &	 &	 &	 &	 &	 &	 &\\
3	 &0	 &27	 &108	 &171	 &136	 &57	 &12	 &1	 &	 &	 &	 &	 &	 &\\
4	 &0	 &81	 &432	 &972	 &1200	 &886	 &400	 &108	 &16	 &1	 &	 &	 &	 &\\
5	 &0	 &243	 &1620	 &4725	 &7920	 &8430	 &5944	 &2810	 &880	 &175	 &20	 &1	 &	 &\\
6	 &0	 &729	 &5832	 &20898	 &44280	 &61695	 &59472	 &40636	 &19824	 &6855	 &1640	 &258	 &24	 &1
\end{array}$
}
\caption{Values of $h(n,k)$ for $0 \leq n\leq 6 $ and $ 0\leq k\leq 13 $.}
\label{Tab:hitch}
\end{table}

\subsection{Overhand knot}\label{subsec:overhand-knot}
\begin{theorem}
The generating polynomial for the $ n $-overhand knot is given by the recurrence relation
\begin{equation}
O_n(x)=(x^2+4x+3)O_{n-1}(x),
\end{equation}
and is expressed by the closed form formula
\begin{equation}
O_n(x)=x\left(x^2+4x+3\right)^n.
\end{equation}
\end{theorem}

\begin{proof}
See \Figs{Fig:1-overhand-states}.
\begin{figure}[ht]
\centering
\includegraphics[width=.7\linewidth]{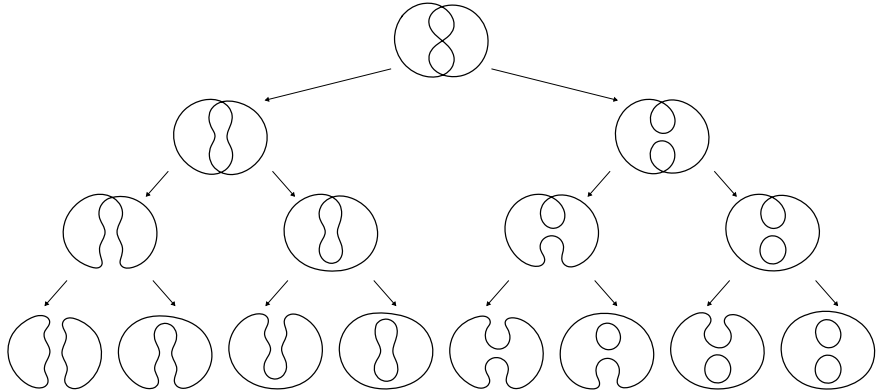}
\caption{The states of the $ 1 $-overhand knot, $ O_1(x)=x^3+4x^2+3x $.}
\label{Fig:1-overhand-states}
\end{figure}
\end{proof}

\begin{corollary}
The generating function for the  sequence $ \big\{O_n(x)\big\}_{n\geq0} $ is given by
\begin{equation}
O(x;y):=\dfrac{x}{1-y\left(x^2+4x+3\right)}.
\end{equation}
\end{corollary}
The results are as expected since the knots $ H_1 $ and $ O_1 $ are equivalent under the $ OS^2 $ move. Thus we obtain the same statistics as the hitch knot (subsection \ref{subsec:hitch-knot}).

\section{The generating polynomial for the closure}\label{Sec:gp2}
\subsection{Preliminaries}
Let $ K $ be a knot diagram. The state diagrams of the closure of $ K $  can be illustrated as in \Figs{fig:stateofclosure}.
\begin{figure}[ht]
\centering
\includegraphics[width=0.6\linewidth]{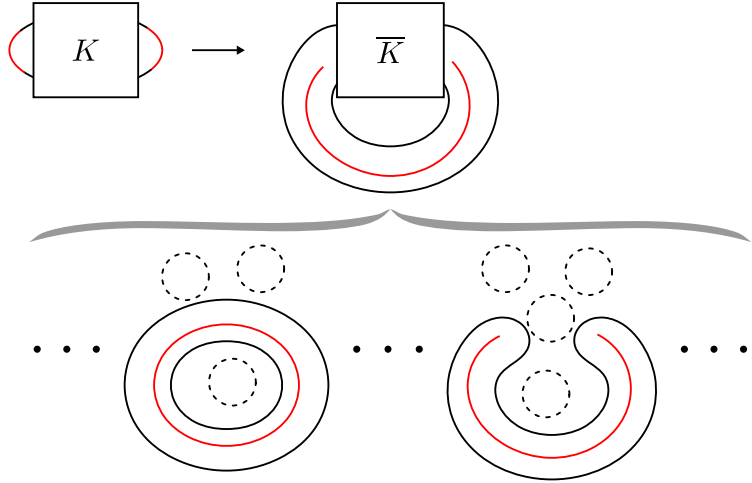}
\caption{The state diagrams of a closure of the knot $ K $.}
\label{fig:stateofclosure}
\end{figure}
Therefore, there exist two polynomials $ \upalpha_K,\upbeta_K\in\mathbb{N}[x]$  such that we can write the generating polynomial for the knot $ \overline{K} $ as
\begin{equation}\label{eq:components}
\overline{K}(x)=x^2\upalpha_K(x)+x\upbeta_K(x).
\end{equation}
Conversely, if we ``cancel'' the connection, then we have the additional identity
\begin{equation}\label{eq:closure}
{K}(x)=x^2\upbeta_K(x) +x\upalpha_K(x).
\end{equation}
Let us refer to $ \upalpha_K(x) $ and $ \upbeta_K(x) $ as the components of the polynomial $ K(x) $. We can then  extend \eqref{eq:components} as follows.
\begin{proposition}\label{prop:sumclosure}
Let $ K $, $ K' $ be two knot diagrams,  and let $ K(x) $, $ K'(x) $ be respectively their generating polynomials. If $ \upalpha_{K}(x) $ and $ \upbeta_{K}(x) $ are the components of the polynomial $ K(x) $, then we obtain
\[
\overline{\left(K\#K'\right)}(x)=\upalpha_{K}(x)\overline{K'}(x)+\upbeta_{K}(x)K'(x).
\]
\end{proposition}
We should note that if $ K'= U$, then we get back to the definition of the closure \eqref{eq:closure}. Indeed we have
\begin{align*}
\overline{K}(x)=\overline{\left(K\# U\right) }(x)&=\upalpha_{K}(x)\overline{U}(x)+\upbeta_{K}(x)U(x)\\
&=x^2\upalpha_{K}(x)+x\upbeta_{K}(x).
\end{align*}

\begin{proof}[Proof of \Prop{prop:sumclosure}] We first need to compute the generating polynomial for the knot $ K $.
Referring to \Figs{fig:stateofclosure}, the ``largest double circle''  becomes the closure  $ \overline{K'} $, while the ``largest single circle'' is then connected with $ K' $. See \Figs{fig:closedsum} for the illustration. 	
\begin{figure}[ht]
\centering
\includegraphics[width=0.625\linewidth]{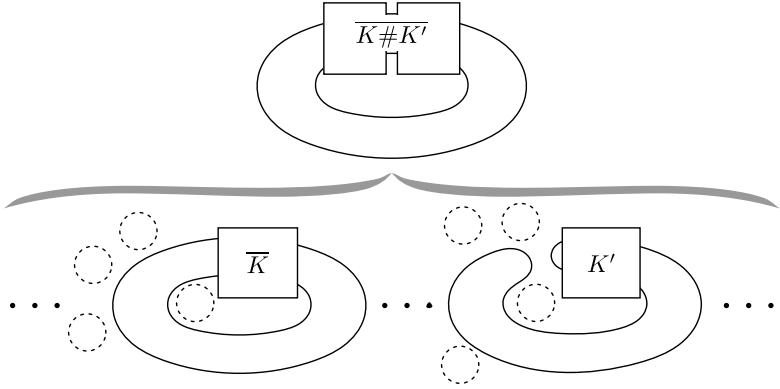}
\caption{The states of a composition closure.}
\label{fig:closedsum}
\end{figure}
Next, we recover the same components $ \upalpha_{K}(x) $ and $ \upbeta_{K}(x) $, and taking into account the previously mentioned substitution, we write
\[
\overline{\left(K\#K'\right)}(x)=\upalpha_{K}(x)\overline{K'}(x)+\upbeta_{K}(x)K'(x).
\]
\end{proof}

\begin{corollary}\label{Cor:GenClosure}
Let $ K(x) $ be the generating polynomial for an arbitrary knot $ K $, and  let $ \upalpha_K(x) $ and $ \upbeta_K(x) $ be its components. For a nonnegative integer $ n $, we have
\begin{align*}
\overline{K_n}(x)&=\left(\upalpha_K(x)+x\upbeta_{K}(x)\right)^n+\left(x^2-1\right)\upalpha_{K}(x)^n\\
&=x^{-1}K_n(x)+\left(x^2-1\right)\upalpha_{K}(x)^n.
\end{align*}
\end{corollary}

\begin{proof}
We write 	$ \overline{K_n}(x)=	\overline{\left(K\#K_{n-1}\right)}(x) $, then apply \Prop{prop:sumclosure}. We have
\begin{align*}
\overline{K_n}(x)&=\upalpha_K(x)\overline{K_{n-1}}(x)+\upbeta_K(x)K_{n-1}\\
&=x^2\upalpha_K(x)^n+\upbeta_K(x)\sum_{k=0}^{n-1}\upalpha_K(x)^{n-k-1}K_{k}(x).
\end{align*}
By \eqref{eq:generatorp} and \eqref{eq:closure} we obtain
\begin{align*}
\overline{K_n}(x)&=x^2\upalpha_K(x)^n+x\upbeta_{K}(x)\dfrac{\upalpha_{K}(x)^n-\left(x^{-1}K(x)\right)^n}{\upalpha_{K}(x)-x^{-1}K(x)}\\
&=x^2\upalpha_K(x)^n+\big(\upalpha_K(x)+x\upbeta_{K}(x)\big)^n-\alpha_K(x)^n\\
&=\big(\upalpha_K(x)+x\upbeta_{K}(x)\big)^n+\left(x^2-1\right)\upalpha_{K}(x)^n\\
&=x^{-1}K_n(x)+\left(x^2-1\right)\upalpha_{K}(x)^n.
\end{align*}
\end{proof}

\begin{corollary}\label{cor:gpclosure}
Let $ \overline{K}(x;y) :=\sum_{n\geq0}\overline{K_n}(x) y^n$ and $ K(x;y) :=\sum_{n\geq0}K_n(x) y^n$ respectively denote the generating function for the sequence $ \left\{\overline{K_n}(x)\right\}_{n\geq 0} $ and $ \big\{K_n(x)\big\}_{n\geq 0} $. Then we have
\begin{align*}
\overline{K}(x;y) &=\dfrac{x^2-1}{1-y\upalpha_K(x)}+\dfrac{x}{x-yK(x;y)}\\
&=\dfrac{1}{1-y\upalpha_K(x)}\Big(x^2+y\upbeta_K(x)K(x;y)\Big).
\end{align*}
\end{corollary}

The key to establishing the generating polynomial for the closure of a generated knot is then to identify  the components of that of the generator. The next results are direct application of Corollaries \ref{Cor:GenClosure} and \ref{cor:gpclosure}.

\subsection{Foil knot} 
Let $ \overline{T_n}(x) :=\sum_{k\geq 0}^{}f(n,k)x^k$ denote the generating polynomial for the $ n $-foil knot. 
\begin{theorem}
The generating polynomial for the $ n $-foil knot is given by the recurrence relation
\begin{equation}\label{eq:Fnx1}
\overline{T_n}(x)=\overline{T_{n-1}}(x)+T_{n-1}(x),
\end{equation}
and is expressed by the closed form formula
\begin{equation}\label{eq:Fnx2}
\overline{T_n}(x)=(x+1)^n+x^2-1.
\end{equation}
\end{theorem}

\begin{proof}
By \Figs{Fig:ClosureT1}, we have  $ \upalpha_{T_1}(x)=1 $ and  $ \upbeta_{T_1}(x)=1 $. Next apply \Cor{Cor:GenClosure}.
\begin{figure}[ht]
\centering
\includegraphics[width=0.25\linewidth]{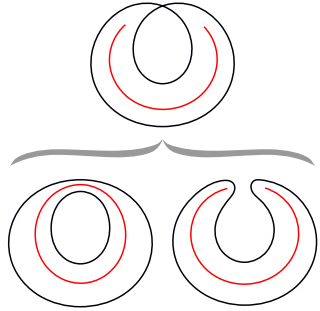}
\caption{The states of the $ 1 $-foil, $ \overline{T_1}(x)=x^2+x $.}
\label{Fig:ClosureT1}
\end{figure}
\end{proof}
\begin{corollary}
The generating function for the  sequence $ \left\{ \overline{T_n} (x)\right\}_{n\geq0} $ is given by
\begin{equation}
\overline{T}(x;y):=\dfrac{1}{1-y}\left(x^2+\dfrac{yx}{1-y(x+1)}\right).
\end{equation}
\end{corollary}

By \eqref{eq:Fnx1} and \eqref{eq:Fnx2}, the coefficients $ f(n,k) $ must satisfy the following recurrence relation:
\begin{equation}
\begin{cases}
f(0,2)=1;\\ 
f(n,0)=0,\ f(n,1)=n,  & n\geq 0;\\
f(n,k)=f(n-1,k)+t(n-1,k),	&k\geq 1,\ n\geq 0.
\end{cases}
\end{equation}
Hence we have \Tabs{Tab:gpfoil} for $ 0\leq n\leq 12 $ and $ 0\leq k\leq 12$.
Also, referring back to \eqref{eq:tnk}, we obtain the following sequences:
\begin{itemize}
\item   $ f(n,1)=n $, the nonnegative integers \cite[\seqnum{A001477}]{Sloane};
\item  $ f(n,2)=\binom{n}{2}+1$ \cite[\seqnum{A152947}]{Sloane};
\item if $ k\geq 3 $, then we obtain the usual  Pascal's triangle \cite[\seqnum{A007318}]{Sloane}.
\end{itemize}

\begin{table}[ht]
\centering
$\begin{array}{c|rrrrrrrrrrrrr}
n\ \backslash\ k	&0	&1	&2	&3	&4	&5	&6	&7	&8	&9	&10\\
\midrule
0	&0	&0	&1	&	&	&	&	&	&	&	&	\\
1	&0	&1	&1	&	&	&	&	&	&	&	&	\\
2	&0	&2	&2	&	&	&	&	&	&	&	&	\\
3	&0	&3	&4	&1	&	&	&	&	&	&	&	\\
4	&0	&4	&7	&4	&1	&	&	&	&	&	&	\\
5	&0	&5	&11	&10	&5	&1	&	&	&	&	&	\\
6	&0	&6	&16	&20	&15	&6	&1	&	&	&	&	\\
7	&0	&7	&22	&35	&35	&21	&7	&1	&	&	&	\\
8	&0	&8	&29	&56	&70	&56	&28	&8	&1	&	&	\\
9	&0	&9	&37	&84	&126	&126	&84	&36	&9	&1	&	\\
10	&0	&10	&46	&120	&210	&252	&210	&120	&45	&10	&1\\
\end{array}$
\caption{Values of $ f(n,k)$ for $ 0\leq n\leq 10 $ and $ 0\leq k\leq 10$.}
\label{Tab:gpfoil}
\end{table}

\begin{remark}
From \Tabs{Tab:twistloop} and \Tabs{Tab:gpfoil} we have 
\begin{equation*}
T_1(x)=\overline{T_1}(x)=x^2+x,
\end{equation*}
from \Tabs{tab:link} and \Tabs{Tab:gpfoil} we read
\begin{equation*}
L_1(x)=\overline{T_2}(x)=2x^2+2x,
\end{equation*}
and finally from \Tabs{Tab:hitch} and \Tabs{Tab:gpfoil} we read
\begin{equation*}
H_1(x)=\overline{T_3}(x)=x^3+4x^2+3x.
\end{equation*}
\end{remark}

\begin{remark}
Following \Rem{Rem:T2T3}, we have the corresponding results on the $ 2n $-foil knot and the $ 3n $-foil knot, namely $ \overline{\left(T_2\right)_n}(x)=\overline{\left(T_{2n}\right)}(x) $ and $ \overline{\left(T_3\right)_n}(x)=\overline{\left(T_{3n}\right)}(x) $.
\paragraph{\boldmath$2n $-foil knot:} let  $ \overline{T_{2n}}(x):=\sum_{k\geq0}^{}f_{2}(n,k)x^k $.
\begin{enumerate}
\item Generating polynomial:  
\begin{equation}
\overline{T_{2n}}(x)=\left(x^2+2x+1\right)^{n}+x^2-1.
\end{equation}
\item Generating function: 
\begin{equation}
\overline{T_{2}}(x;y):=\dfrac{1}{1-y}\left(x^2+\dfrac{yx(x+2)}{1-y(x^2+2x+1)}\right).
\end{equation}
\item Distribution of $ f_2(n,k) $: see \Tabs{Tab:dank}.
\begin{equation}
\begin{cases}
f_{2}(0,2)=1;\\
f_{2}(n,0)=0,f_{2}(n,1)=2n, & n\geq 0;\\
f_{2}(n,k)=f_{2}(n-1,k)+t_{2}(n-1,k-1)+2t_{2}(n-1,k),	&k\geq 1,\ n\geq 0.
\end{cases} 
\end{equation}

\begin{table}[ht]
\centering
$\begin{array}{c|rrrrrrrrrrrrrrr}
n\ \backslash\ k	&0	&1	&2	&3	&4	&5	&6	&7	&8	&9	&10	&11	&12&13&14\\
\midrule
0	&0	&0	&1	&	&	&	&	&	&	&	&	&	&&&\\
1	&0	&2	&2	&	&	&	&	&	&	&	&	&	&&&\\
2	&0	&4	&7	&4	&1	&	&	&	&	&	&	&	&&&\\
3	&0	&6	&16	&20	&15	&6	&1	&	&	&	&	&	&&&\\
4	&0	&8	&29	&56	&70	&56	&28	&8	&1	&	&	&	&&&\\
5	&0	&10	&46	&120	&210	&252	&210	&120	&45	&10	&1	&	&&&\\
6	&0	&12	&67	&220	&495	&792	&924	&792	&495	&220	&66	&12	&1&&\\
7	&0	&14	&92	&364	&1001	&2002	&3003 & 3432& 	3003 &2002	&1001	&364	&91	&14	&1\\
\end{array}$
\caption{Values of $ f_{2}(n,k)$ for $ 0\leq n\leq 7 $ and $ 0\leq k\leq 14$.}
\label{Tab:dank}
\end{table}
\begin{itemize}
\item   $ f_{2}(n,1)=2n $, the nonnegative even numbers, \cite[\seqnum{A005843}]{Sloane};
\item  $f_{2}(n,2)=2n^2 - n + 1$, the maximum number of regions determined by $ n  $ bent lines \cite[\seqnum{A130883}]{GKP,Sloane};
\item if $ k\geq 3 $, then we obtain  $f_{2}(n,k)=\binom{2n}{k}  $, the even-numbered rows of Pascal's triangle \cite[\seqnum{A034870}]{Sloane}.
\end{itemize}
\end{enumerate}

\paragraph{\boldmath$3n $-foil knot:} let $ \overline{T_{3n}}(x):=\sum_{k\geq0}^{}f_{3}(n,k)x^k $.
\begin{enumerate}
\item Generating polynomial:  
\begin{equation}
\overline{T_{3n}}(x)=\left(x^3+3x^2+3x+1\right)^{n}+x^2-1.
\end{equation}
\item Generating function: 
\begin{equation}
\overline{T_{3}}(x;y):=\dfrac{1}{1-y}\left(x^2+\dfrac{yx(x^2+3x+3)}{1-y(x^3+3x^2+3x+1)}\right).
\end{equation}
\item Distribution of $ f_3(n,k) $: see \Tabs{Tab:ddnk}.
\begin{equation}
\begin{cases}
f_{3}(0,2)=1;\\
f_{3}(n,0)=0,\ f_{3}(n,1)=3n,\ f_{3}(n,2)=\dfrac{9n^2 - 3n + 2}{2}, & n\geq 0;\\
f_{3}(n,k)=f_{3}(n-1,k)+t_3(n-1,k-2)\\
\hphantom{f_3(n,k)=}+3t_3(n-1,k-1)+3t_3(n-1,k),	&k\geq 2,\ n\geq 0.
\end{cases}
\end{equation}

\begin{table}[ht]
\centering
\resizebox{\linewidth}{!}{%
$\begin{array}{c|rrrrrrrrrrrrrrrr}
n\ \backslash\ k	&0	&1	&2	&3	&4	&5	&6	&7	&8	&9	&10	&11	&12&13&14&15\\
\midrule
0	&0	&0	&1	&	&	&	&	&	&	&	&	&	& &&&\\
1	&0	&3	&4	&1	&	&	&	&	&	&	&	&	&&&&\\
2	&0	&6	&16	&20	&15	&6	&1	&	&	&	&	&	&&&&\\
3	&0	&9	&37	&84	&126	&126	&84	&36	&9	&1	&	&	&&&&\\
4	&0	&12	&67	&220	&495	&792	&924	&792	&495	&220	&66	&12	&1&&&\\
5	&0	&15	&106	&455	&1365	&3003	&5005	&6435	&6435	&5005	&3003	&1365	&455 &105&15&1\\
\end{array}$
}
\caption{Values of $ f_{3}(n,k)$ for $ 0\leq n\leq 5 $ and $ 0\leq k\leq 15$.}
\label{Tab:ddnk}
\end{table}

\begin{itemize}
\item   $ f_{3}(n,1)=3n $, the multiples of 3 \cite[\seqnum{A008585}]{Sloane};
\item  $f_{3}(n,2)=\dfrac{9n^2 - 3n + 2}{2}$,  the generalized polygonal numbers \cite[\seqnum{A080855}]{Sloane};
\item  $f_{3}(n,3)=\dfrac{n(3n - 1)(3n - 2)}{2}$,  the dodecahedral numbers \cite[\seqnum{A006566}]{Sloane};
\item  $f_{3}(n,2n+1)=\binom{3n}{n-1}$ \cite[\seqnum{A004319}]{Sloane};
\item in fact, if $ k\geq 3 $, then $ f_3(n,k)=t_3(n,k+1)=\binom{3n}{k} $ \cite[$ \seqnum{A007318}(3n,k) $]{Sloane}.
\end{itemize}
\end{enumerate}
\end{remark}

\subsection{Chain Link}
Let $ \overline{L_n}(x):=\sum_{k\geq 0}^{}c(n,k)x^k$ denote the generating polynomial for the $ n $-chain link. 
\begin{theorem}
The generating polynomial for the $ n $-chain link is given by the recurrence relation
\begin{equation}\label{eq:Cnx1}
\overline{L_n}(x)=(x+2)\overline{L_{n-1}}(x)+L_{n-1}(x),
\end{equation}
and is expressed by the closed form formula
\begin{equation}\label{eq:Cnx2}
\overline{L_n}(x)=(2x+2)^n+(x^2-1)(x+2)^n.
\end{equation}
\end{theorem}

\begin{proof}
By \Figs{Fig:ClosureL1}  we get $ \upalpha_{L_1}(x)=x+2 $ and  $ \upbeta_{L_1}(x)=1 $. 
\begin{figure}[ht]
\centering
\includegraphics[width=0.5\linewidth]{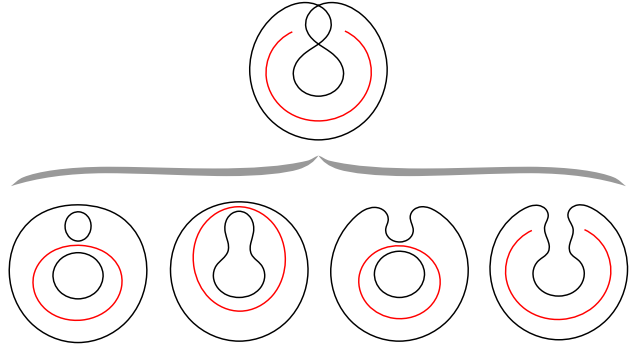}
\caption{The states of the $ 1 $-chain link, $ \overline{L_1}(x)=x^3+2x^2+x $.}
\label{Fig:ClosureL1}
\end{figure}
\end{proof}

\begin{corollary}
The generating function for the  sequence $ \left\{ \overline{L_{n}} (x)\right\}_{n\geq0} $ is given by
\begin{equation}
\overline{L}(x;y):=\dfrac{1}{1-y(x+2)}\left(x^2+\dfrac{yx}{1-y(2x+2)}\right)\cdot
\end{equation}
\end{corollary}
Combining \eqref{eq:Cnx1} and \eqref{eq:Cnx2}, we have
\begin{equation}
\begin{cases}
c(0,2)=1;\\
c(n,0)=0,\ c(n,1)=n2^{n-1}, & n\geq 0;\\
c(n,k)=c(n-1,k-1)+2c(n-1,k)+\ell(n-1,k),	&k\geq 1,\ n\geq 0.
\end{cases}
\end{equation}
\Tabs{tab:nlinkchain} gives the values of $ c(n,k) $ for $ 0\leq n\leq 9 $ and $ 0\leq k\leq 11 $ \cite[\seqnum{ A300184}]{Sloane}. Next, we identify the following integer sequences:
\begin{itemize}
\item $ c(n,1)= n 2^{n-1}$ \cite[\seqnum{A001787}]{Sloane}; 	
\item $ c(n,2)= \left(3n^2-3n+8\right)2^{n-3}$ \cite[\seqnum{A300451}]{Sloane}; 
\item $ c(n,n) = 2n(n-1)+2^n-1 $ \cite[\seqnum{A295077}]{Sloane};  
\item $ c(n,n+1)=2n$, the nonnegative even numbers \cite[\seqnum{A005843}]{Sloane};  	
\item $ c(n,n+2) =1$,  the all $ 1 $'s sequence \cite[\seqnum{A000012}]{Sloane}.
\end{itemize}
\begin{table}[ht]
\centering
$\begin{array}{c|rrrrrrrrrrrr}
n\ \backslash\ k		 &0		 &1		 &2		 &3		 &4		 &5		 &6		 &7		 &8		 &9		 &10		 &11\\
\midrule
0	 &0	 &0	 &1	 &	 &	 &	 &	 &	 &	 &	 &	 &\\
1	 &0	 &1	 &2	 &1	 &	 &	 &	 &	 &	 &	 &	 &\\
2	 &0	 &4	 &7	 &4	 &1	 &	 &	 &	 &	 &	 &	 &\\
3	 &0	 &12	 &26	 &19	 &6	 &1	 &	 &	 &	 &	 &	 &\\
4	 &0	 &32	 &88	 &88	 &39	 &8	 &1	 &	 &	 &	 &	 &\\
5	 &0	 &80	 &272	 &360	 &1230	 &71	 &10	 &1	 &	 &	 &	 &\\
6	 &0	 &192	 &784	 &1312	 &1140	 &532	 &123	 &12	 &1	 &	 &	 &\\
7	 &0	 &448	 &2144	 &4368	 &4872	 &3164	 &1162	 &211	 &14	 &1	 &	 &\\
8	 &0	 &1024	 &5632	 &13568	 &18592	 &15680	 &8176	 &2480	 &367	 &16	 &1	 &\\
9	 &0	 &2304	 &14336	 &39936	 &65088	 &67872	 &46368	 &20304	 &5262	 &655	 &18	 &1
\end{array}$
\caption{Values of $ c(n,k) $ for $ 0\leq n\leq 9 $ and $ 0\leq k\leq 11 $.}
\label{tab:nlinkchain}
\end{table}

\begin{remark}
From \Tabs{Tab:twistloop} and \Tabs{tab:nlinkchain} we have 
\begin{equation*}
T_2(x)=\overline{L_1}(x)=x^3+2x^2+x,
\end{equation*}
and from \Tabs{Tab:gpfoil} and \Tabs{tab:nlinkchain} we read
\begin{equation*}
\overline{T_4}(x)=\overline{L_2}(x)=x^4+4x^3+7x^2+4x.
\end{equation*}
\end{remark}

\subsection{Twist bracelet}\label{subsec:twistedbracelet}
Let $ \overline{W_n}(x):=\sum_{k\geq 0}^{}b(n,k)x^k$ denote the generating polynomial for the $ n $-twist bracelet. 
\begin{theorem}
The generating polynomial for the $ n $-twist bracelet is given by the recurrence relation
\begin{equation}\label{eq:Bnx1}
\overline{W_{n}} (x)=(x+2)\overline{W_{n-1}}(x)+(2x+3)W_{n-1}(x)\\
\end{equation}
and is expressed by the closed form formula
\begin{equation}\label{eq:Bnx2}
\overline{W_{n}} (x)=\left(2x^2+4x+2\right)^n+\left(x^2-1\right)(x+2)^n.
\end{equation}
\end{theorem}

\begin{proof}
We have  $ \upalpha_{W_1}(x)=x+2 $ and  $ \upbeta_{W_1}(x)=2x+3 $, see \Figs{Fig:ClosureW1}.
\begin{figure}[ht]
\centering
\includegraphics[width=.9\linewidth]{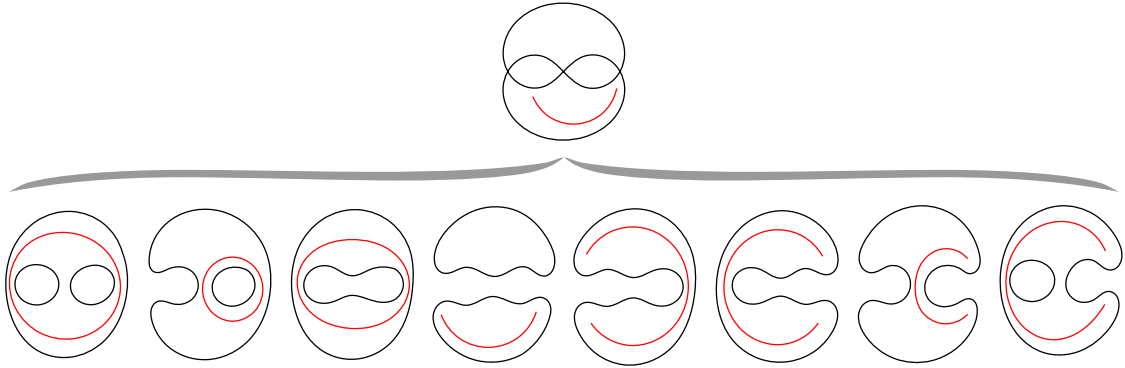}
\caption{The states of the  $ 1 $-twist bracelet, $ \overline{W_{1}}(x)=x^3+4x^2+3x $.}
\label{Fig:ClosureW1}
\end{figure}
\end{proof}

\begin{corollary}
The generating function for the  sequence $ \left\{ \overline{W_{n}} (x)\right\}_{n\geq0} $ is given by
\begin{equation*}
\overline{W}(x;y):=\dfrac{1}{1-y(x+2)}\left(x^2+\dfrac{yx(2x+3)}{1-y(2x^2+4x+2)}\right).
\end{equation*}
\end{corollary}
By \eqref{eq:wnk}, \eqref{eq:Bnx1} and \eqref{eq:Bnx2}   we have the following recurrence
\begin{equation}
\begin{cases}
b(0,2)=1;\\
b(n,0)=0,\  b(n,1)= 3n2^{n-1}, & n\geq 0;\\
b(n,k)=b(n-1,k-1)+2b(n-1,k)\\
\hphantom{b(n,k)=}+2w(n-1,k-1)+3w(n-1,k),	&k\geq 1,\ n\geq 0.
\end{cases}
\end{equation}
\Tabs{Tab:Bracelet} gives some of the values of $ b(n,k) $. We collect the following sequences:
\begin{itemize}
\item $ b(n,1)=3n2^{n-1} $ \cite[\seqnum{A167667}]{Sloane};
\item $ b(n,2n) = 2^n $ except $ b(1,2) = 4 $ and $ b(2,4) = 5 $, the  independence number of Keller graphs \cite[\seqnum{A258935}]{Sloane}.	 
\end{itemize}
\begin{table}[ht]
\centering
\resizebox{\linewidth}{!}{%
$\begin{array}{c|rrrrrrrrrrrrr}
n\ \backslash\ k		 &0		 &1		 &2		 &3		 &4		 &5		 &6		 &7		 &8		 &9		 &10		 &11	 &12\\
\midrule
0	 &0	 &0	 &1	 &	 &	 &	 &	 &	 &	 &	 &	 &\\
1	 &0	 &3	 &4	 &1	 &	 &	 &	 &	 &	 &	 &	 &\\
2	 &0	 &12	 &27	 &20	 &5	 &	 &	 &	 &	 &	 &	 &\\
3	 &0	 &36	 &122	 &171	 &126	 &49	 &8	 &	 &	 &	 &	 &\\
5	 &0	 &96	 &440	 &920	 &1143	 &904	 &449	 &128	 &16	 &	 &	 &\\
6	 &0	 &240	 &1392	 &3880	 &6790	 &8103	 &6730	 &3841	 &1440	 &320	 &32	 &\\
7	 &0	 &576	 &4048	 &14112	 &31860	 &50836	 &59195	 &50700	 &31681	 &14080	 &4224	 &768	 &64
\end{array}$}
\caption{Values of $ b(n,k) $ for $ 0\leq n\leq 7 $ and $ 0\leq k\leq 12 $.}
\label{Tab:Bracelet}
\end{table}

\begin{remark}
From \Tabs{Tab:hitch}, \Tabs{Tab:gpfoil} and \Tabs{Tab:Bracelet} we have 
\begin{equation*}
H_1(x)=\overline{T_3}(x)=\overline{W_1}(x)=x^3+4x^2+3x.
\end{equation*}
\end{remark}

\subsection{Ringbolt hitching}
Let $ \overline{H_n}(x):=\sum_{k\geq 0}^{}r(n,k)x^k$ denote the generating polynomial for the $ n $-ringbolt hitching. 

\begin{theorem}
The generating polynomial for the $ n $-ringbolt hitching knot is given by the recurrence relation
\begin{equation}\label{Eq:Rnx1}
\overline{H_n}(x)=(2x+3)\overline{H_{n-1}}(x)+(x+2)H_{n-1}(x).
\end{equation}
and is expressed by the closed form formula
\begin{equation}\label{Eq:Rnx2}
\overline{H_n}(x)=\left(x^2+4x+3\right)^n+\left(x^2-1\right)(2x+3)^n.
\end{equation}
\end{theorem}

\begin{proof}
Proceeding with the usual fashion we obtain  $ \upalpha_{H_1}(x)=2x+3 $ and  $ \upbeta_{H_1}(x)=x+2 $, see \Figs{Fig:ClosedH1}. 
\begin{figure}[ht]
\centering
\includegraphics[width=.9\linewidth]{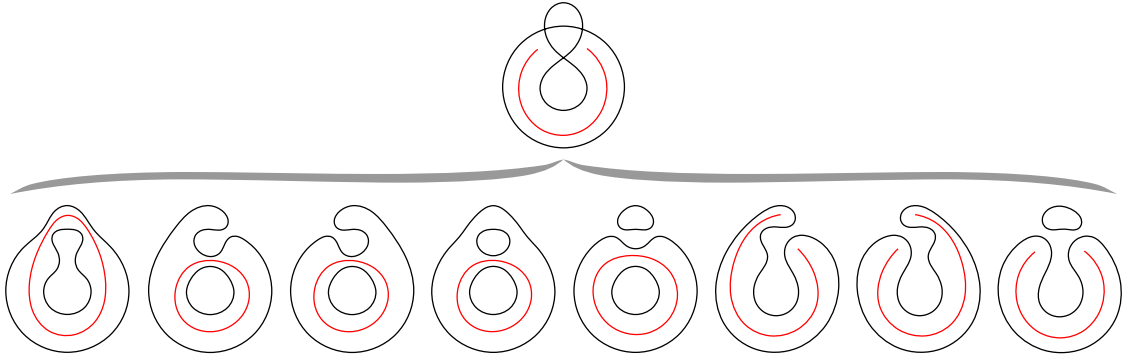}
\caption{The states of the $ 1 $-ringbolt hitching, $\overline{H_1}(x)=2x^3+4x^2+2x $.}
\label{Fig:ClosedH1}
\end{figure}
\end{proof}

\begin{corollary}
The generating function for the  sequence $ \left\{\overline{H_n}(x)\right\}_{n\geq0} $ is given by
\begin{equation}
\overline{H}(x;y):=\dfrac{1}{1-y(2x+3)}\left(x^2+\dfrac{yx(x+2)}{1-y(x^2+4x+3)}\right)
\end{equation}
\end{corollary}

Combining \eqref{Eq:Rnx1} with \eqref{Eq:Rnx2}, we obtain the recurrence relation
\begin{equation}
\begin{cases}
r(0,2)=1;\\
r(n,0)=0,\ r(n,1)=2n3^{n-1}, & n\geq 0;\\
r(n,k)=2r(n-1,k-1)+3r(n-1,k)\\
\hphantom{r(n,k)=} +h(n-1,k-1)+2h(n-1,k), &  k\geq 1,\ n\geq 0.
\end{cases}
\end{equation}
We then obtain the values of  $ r(n,k) $ for $ 0\leq n\leq 6 $ and $ 0\leq k\leq 12 $ as given in \Tabs{tab:ringbolt}. We only recognize here the sequence defined by $ r(n,1)=2n3^{n-1}$ \cite[\seqnum{A212697}]{Sloane}.

\begin{table}[ht]
\centering
$\begin{array}{c|rrrrrrrrrrrrr}
n\ \backslash\ k		 &0		 &1		 &2		 &3		 &4		 &5		 &6		 &7		 &8		 &9		 &10		 &11		 &12\\
\midrule
0	 &0	 &0	 &1	 &	 &	 &	 &	 &	 &	 &	 &	 &	 &\\
1	 &0	 &2	 &4	 &2	 &	 &	 &	 &	 &	 &	 &	 &	 &\\
2	 &0	 &12	 &27	 &20	 &5	 &	 &	 &	 &	 &	 &	 &	 &\\
3	 &0	 &54	 &162	 &182	 &93	 &20	 &1	 &	 &	 &	 &	 &	 &\\
4	 &0	 &216	 &837	 &1320	 &1086	 &496	 &124	 &16	 &1	 &	 &	 &	 &\\
5	 &0	 &810	 &3888	 &8010	 &9270	 &6632	 &3050	 &912	 &175	 &20	 &1	 &	 &\\
6	 &0	 &2916	 &16767	 &42876	 &64395	 &63216	 &42732	 &20400	 &6919	 &1640	 &258	 &24	 &1
\end{array}$
\caption{Values of $ r(n,k) $ for $ 0\leq n\leq 6 $ and $ 0\leq k\leq 12 $.}
\label{tab:ringbolt}
\end{table}

\begin{remark}
Reading \Tabs{Tab:Twistlink} and \Tabs{tab:ringbolt} yields 
\begin{equation*}
W_3(x)=\overline{H_1}(x)=2x^3+4x^2+2x.
\end{equation*}
\end{remark}

\subsection{Sinnet of square knotting}
Let $ \overline{O_n}(x):=\sum_{k\geq 0}^{}s(n,k)x^k$ denote the generating polynomial for the $ n $-sinnet of square knotting. 

\begin{theorem}
The generating polynomial for the $ n $-sinnet of square knotting is given by the recurrence relation
\begin{equation}\label{eq:Snx1}
\overline{O_n}(x)=\left(x^2+3x+3\right)\overline{O_{n-1}}(x)+O_{n-1}(x),
\end{equation}
and is expressed by the closed form formula
\begin{equation}\label{eq:Snx2}
\overline{O_n}(x)=\left(x^2+4x+3\right)^n+\left(x^2-1\right)\left(x^2+3x+3\right)^n.
\end{equation}
\end{theorem}

\begin{proof}
Again, the usual routine allows us to give both expressions of $ O_n(x) $. By \Figs{Fig:ClosureO1}, we have  $ \upalpha_{O_1}(x)=x^2+3x+3 $ and $ \upbeta_{O_1}(x)=1 $. We conclude by \Cor{Cor:GenClosure}.
\begin{figure}[ht]
\centering
\includegraphics[width=.9\linewidth]{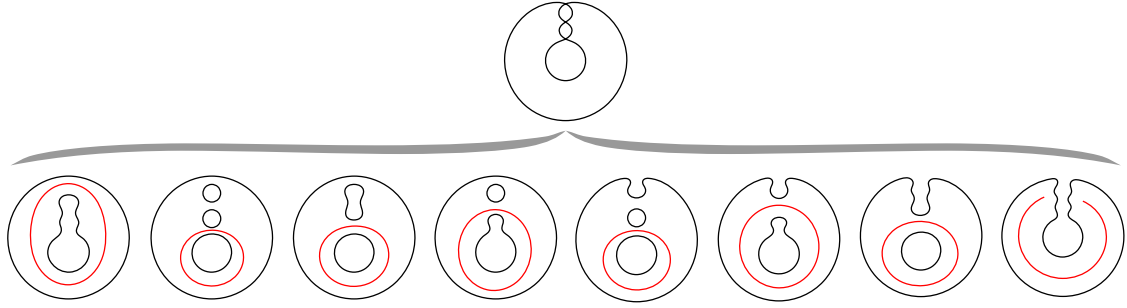}
\caption{The states of the $ 1 $-sinnet of square knotting, $ \overline{O_1}(x)=x^4+3x^3+3x^2+x $.}
\label{Fig:ClosureO1}
\end{figure}
\end{proof}

\begin{corollary}
The generating function for the  sequence $ \left\{\overline{O_n}(x)\right\}_{n\geq0} $ is given by
\begin{equation}
\overline{O}(x;y):=\dfrac{1}{1-y(x^2+3x+3)}\left(x^2+\dfrac{yx}{1-y(x^2+4x+3)}\right).
\end{equation}
\end{corollary}
Combining  \eqref{eq:hnk}, \eqref{eq:Snx1} and \eqref{eq:Snx2}  yields
\begin{equation}
\begin{cases}
s(0,2)=1;\\
s(n,0)=0,\  s(n,1)=n3^{n-1},\ s(n,2)=3^n+7\dfrac{n(n-1)}{2}3^{n-2}, & n\geq 0;\\
s(n,k)=s(n-1,k-2)+3s(n-1,k-1)+3s(n-1,k)\\
\hphantom{s(n,k)=}+h(n-1,k),  &k\geq 2,\ n\geq 0.
\end{cases}
\end{equation}
In \Tabs{tab:sinnet}, we list the values of $ s(n,k) $ for small $ n $ and $ k $. We recognize the following sequences:
\begin{itemize}
\item $ s(n,1) = n3^{n-1} $ \cite[\seqnum{A027471}]{Sloane};
\item $ s(n,2n+1) =3n$,  the multiples of $ 3 $ \cite[\seqnum{A008585}]{Sloane};
\item $ s(n,2n-1) =\dfrac{n(3n - 1)(3n - 2)}{2}$,  the dodecahedral numbers \cite[\seqnum{A006566}]{Sloane};
\item $ s(n,2n) =\dfrac{3n(3n-1)}{2}$, three times pentagonal numbers \cite[\seqnum{A062741}]{Sloane}.
\end{itemize}

\begin{table}[ht]
\centering
$\begin{array}{c|rrrrrrrrrrrrr}
n\ \backslash\ k		 &0		 &1		 &2		 &3		 &4		 &5		 &6		 &7		 &8		 &9		 &10		 &11		 &12\\
\midrule
0	 &0	 &0	 &1	 &	 &	 &	 &	 &	 &	 &	 &	 &	 &\\
1	 &0	 &1	 &3	 &3	 &1	 &	 &	 &	 &	 &	 &	 &	 &\\
2	 &0	 &6	 &16	 &20	 &15	 &6	 &1	 &	 &	 &	 &	 &	 &\\
3	 &0	 &27	 &90	 &136	 &129	 &84	 &36	 &9	 &1	 &	 &	 &	 &\\
4	 &0	 &108	 &459	 &876	 &1021	 &832	 &501	 &220	 &66	 &12	 &1	 &	 &\\
5	 &0	 &405	 &2133	 &5085	 &7350	 &7321	 &5420	 &3103	 &1375	 &455	 &105	 &15	 &1
\end{array}$
\caption{Values of $ s(n,k) $ for $ 0\leq n\leq 5 $ and $ 0\leq k\leq 12 $.}
\label{tab:sinnet}
\end{table}

\begin{remark}
From \Tabs{Tab:twistloop} and \Tabs{tab:sinnet} we have 
\begin{equation*}
T_3(x)=\overline{O_1}(x)=x^4+3x^3+3x^2+x,
\end{equation*}
and from \Tabs{Tab:gpfoil} and \Tabs{tab:sinnet} we have 
\begin{equation*}
\overline{T_6}(x)=\overline{O_2}(x)=x^6+6x^5+15x^4+20x^3+16x^2+6x.
\end{equation*}
\end{remark}

\subsection{Twist knot}
A \textit{twist knot} is a knot  obtained by repeatedly twisting a closed loop, and then linking the ends together \cite[p.\ 57]{JH}. We let $ \mathcal{T}_n $ denote a  \textit{$ n$-twist knot}, i.e., a twist knot of  $ n $  half twists \cite{Ramaharo}. Examples of twist knots are given in \Figs{fig:twist-knot}.
\begin{figure}[ht]
\centering
\includegraphics[width=1\linewidth]{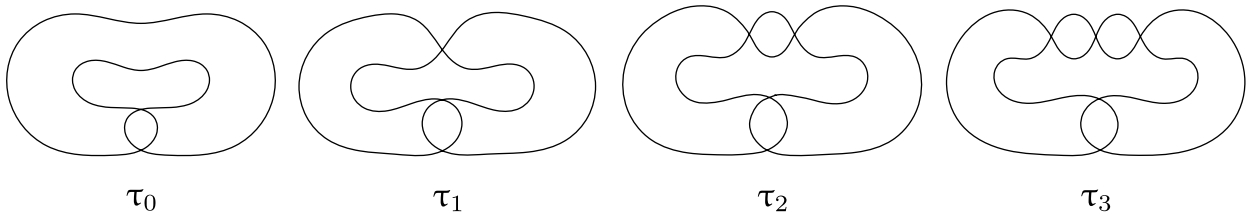}
\caption{$n $-twist knots, $ n=0,1,2,3 $.}
\label{fig:twist-knot}
\end{figure}

Let then $ \mathcal{T}_{n}(x):=\sum_{k\geq 0}^{}\tau(n,k)x^k$ denote the generating polynomial for the $ n $-twist knot. 
\begin{theorem}
The generating polynomial for the $ n $-twist knot is  given by the relation
\begin{equation}
\mathcal{T}_n(x)=(x+2)\overline{T_n}(x)+T_n(x),
\end{equation}
and has the following closed form
\begin{equation}
\mathcal{T}_n(x)=2(x+1)^{n+1}+x^3+2x^2-x-2.
\end{equation}
\end{theorem}

\begin{proof}
The $ n $-twist knot can be decomposed into the closure of the connected sum of the Hopf link and the $ n $-twist loop, i.e., $ \mathcal{T}_n=\overline{L_1\#T_n} $, see \Figs{fig:twist-knot-decomposition}.
\begin{figure}[ht]
\centering
\includegraphics[width=.5\linewidth]{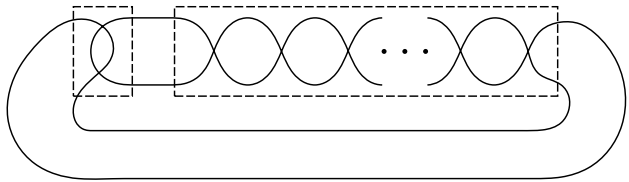}
\caption{The twist knot as the closed connected sum of the Hopf link and the twist loop.}
\label{fig:twist-knot-decomposition}
\end{figure}	

Applying \Prop{prop:sumclosure}, we have
\[
\mathcal{T}_n(x)=\upalpha_{L_1}(x)\overline{T_n}(x)+\upbeta_{L_1}(x)T_n(x),
\]
where $ \upalpha_{L_1}(x) =x+2$ and $ \upbeta_{L_1}(x) =1$  are the components of $ L_1(x) $. The closed form immediately follows.
\end{proof}

\begin{corollary}
The generating function for the  sequence $ \left\{\mathcal{T}_n(x)\right\}_{n\geq 0} $ is given by
\begin{equation*}
\mathcal{T}(x;y)=\dfrac{2x+2}{1-y(x+1)}+\dfrac{x^3+2x^2-x-2}{1-y}\cdot
\end{equation*}
\end{corollary}

Now we can draw up the usual table of coefficients and the collect the OEIS A-records using the following recurrence:
\[\begin{cases}
\tau(n,0)=0,\ \tau(n,1)=2n+1, 				& n\geq 0;\\
\tau(n,k)=f(n,k-1)+2f(n,k)+t(n,k),		&k\geq 1,\ n\geq 0.
\end{cases}\]

\begin{table}[ht]
\centering
$\begin{array}{c|rrrrrrrrrrrrr}
n\ \backslash\ k		 &0		 &1		 &2		 &3		 &4		 &5		 &6		 &7		 &8		 &9		 &10		 &11		 &12\\
\midrule
0	 &0	 &1	 &2	 &1	 &	 &	 &	 &	 &	 &	 &	 &	 &\\
1	 &0	 &3	 &4	 &1	 &	 &	 &	 &	 &	 &	 &	 &	 &\\
2	 &0	 &5	 &8	 &3	 &	 &	 &	 &	 &	 &	 &	 &	 &\\
3	 &0	 &7	 &14	 &9	 &2	 &	 &	 &	 &	 &	 &	 &	 &\\
4	 &0	 &9	 &22	 &21	 &10	 &2	 &	 &	 &	 &	 &	 &	 &\\
5	 &0	 &11	 &32	 &41	 &30	 &12	 &2	 &	 &	 &	 &	 &	 &\\
6	 &0	 &13	 &44	 &71	 &70	 &42	 &14	 &2	 &	 &	 &	 &	 &\\
7	 &0	 &15	 &58	 &113	 &140	 &112	 &56	 &16	 &2	 &	 &	 &	 &\\
8	 &0	 &17	 &74	 &169	 &252	 &252	 &168	 &72	 &18	 &2	 &	 &	 &\\
90	 &0	 &19	 &92	 &241	 &420	 &504	 &420	 &240	 &90	 &20	 &2	 &	 &\\
10	 &0	 &21	 &112	 &331	 &660	 &924	 &924	 &660	 &330	 &110	 &22	 &2	 &\\
11	 &0	 &23	 &134	 &441	 &990	 &1584	 &1848	 &1584	 &990	 &440	 &132	 &24	 &2
\end{array}$
\caption{Values of $ \tau(n,k) $ for $ 0\leq n\leq 11 $ and $ 0\leq k\leq 12 $.}
\label{tab:twistknot}
\end{table}

\begin{itemize}
\item $ \tau(n,1)=2n+1$, the odd numbers \cite[\seqnum{A005408}]{Sloane};
\item $ \tau(n,2)=n^2+n+2 $, the maximum number of regions into which the plane is divided by $ n+1 $ circles \cite[\seqnum{A014206}]{Johnson};
\item  $\tau(n,3)=\frac{1}{3}\big(n^3-n+3\big)$ \cite[\seqnum{A064999}]{Sloane};
\item if $ k\geq 4 $, then  $ \left(\tau(n,k)\right)_{n\geq 3}$ is a horizontal shifted twice Pascal's triangle \cite[\seqnum{A028326}]{Sloane}.
\end{itemize}

\begin{remark}
From \Tabs{Tab:twistloop} and \Tabs{tab:twistknot} we have 
\begin{equation*}
T_2(x)=\mathcal{T}_0(x)=x^3+2x^2+x,
\end{equation*}
and from \Tabs{Tab:hitch}, \Tabs{Tab:gpfoil} and \Tabs{tab:twistknot} we read
\begin{equation*}
H_1(x)=\overline{T_3}(x)=\mathcal{T}_1(x)=x^3+4x^2+3x.
\end{equation*}
\end{remark}

\subsection{The alternative closures}
We first introduce the following notation.
\begin{notation} Let
\begin{itemize}
\item $ \mathcal{S}_1\#\mathcal{S}_1:=\left\{K\#K'\mid \left(K,K'\right)\in\mathcal{S}_1\times\mathcal{S}_1\right\} $;
\item $ \mathcal{S}_1\#\mathcal{S}_1\#\mathcal{S}_1:=\left\{K\#K'\#K''\mid \left(K,K',K''\right)\in\mathcal{S}_1\times\mathcal{S}_1\times\mathcal{S}_1\right\} $;
\item $ \mathcal{S}_1\#\mathcal{S}_{2,1}:=\left\{K\#K'\mid \left(K,K'\right)\in\mathcal{S}_1\times\mathcal{S}_{2,1}\right\} $. 
\end{itemize}
Then we have
\begin{itemize}
\item $ \mathcal{S}_{2,2} = \mathcal{S}_1\#\mathcal{S}_1 $;
\item $ \mathcal{S}_{3,1}\cup\mathcal{S}_{3,2}= \mathcal{S}_1\#\mathcal{S}_1\#\mathcal{S}_1 $;
\item$ \mathcal{S}_{3,3} = \mathcal{S}_1\#\mathcal{S}_{2,1} $.
\end{itemize}
\end{notation}

Moreover, writing $ K=U\#K $ allows us to assume that the elementary knot $ K $ can be decomposed in a way that one might disconnect some knot factors that are not taken into consideration when connecting with a copy of the actual knot. For the sake of clarity, let us  use the asterisk sign $ * $  to indicate that none of the arcs of the concerned knot are  involved when generating $ K_n $ and $ \overline{K_n} $. For example, $ K=T_1\#L_1^* $ means that  the connected sum is performed along some section of $ T_1 $, and $ L_1 $ might be disconnected. We distinguish the following excluding cases:
\begin{itemize}
\item if we cannot disconnect an  elementary knot, then $\overline{K_n} \in\left\{ \overline{T_n} ,  \overline{T_{2n}} ,  \overline{T_{3n}} , \overline{L_{n}} ,   \overline{W_{n}} , \overline{H_n}, \overline{O_n}\right\} $;
\item if $ K=U\#K^* $  with $ K^*\in\left\{T_1^*,T_1^*\#T_1^*, T_1^*\#T_1^*\#T_1^*,L_1^*\#T_1^*,L_1^*,H_1^*,O_1^*\right\} $,

then $ \overline{K_n}(x)=\left(U^2\#K_n\right)(x)$ and $ \overline{K}(x;y)=xK(x;y) $, see \Figs{Fig:closure-alternative};
\begin{figure}[ht]
\centering
\hspace*{\fill}
\subfigure[]{\includegraphics[width=0.2\linewidth]{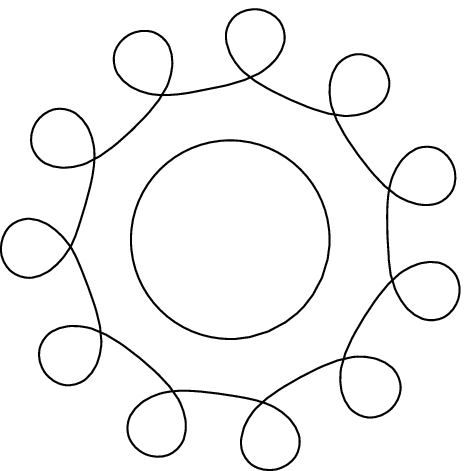}\label{subfig:ul11}}\hfill%
\subfigure[]{\includegraphics[width=0.2\linewidth]{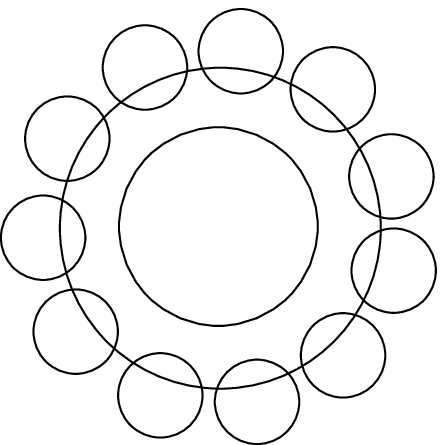}\label{subfig:ut11}}\hfill%
\subfigure[]{\includegraphics[width=0.2\linewidth]{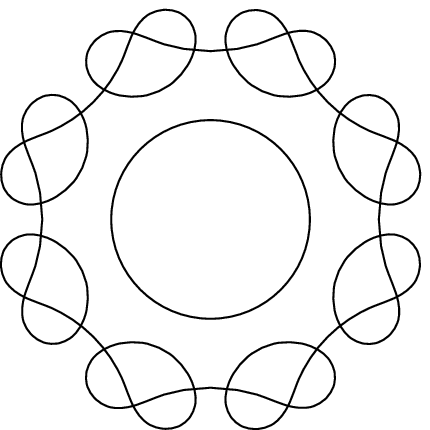}\label{subfig:uh8}}\hfill
\subfigure[]{\includegraphics[width=0.2\linewidth]{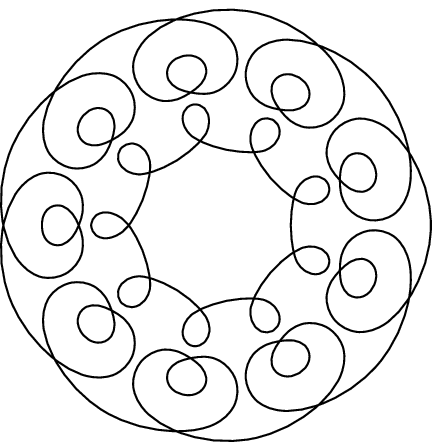}\label{subfig:ut1t1t1}}
\hspace*{\fill}
\caption{Examples of closed connected sum along the section of the unknot: \subref{subfig:ul11} $\overline{\left(U\#T_1^*\right)_{11}}$;
\subref{subfig:ut11} $\overline{\left(U\#L_1^*\right)_{11}}$;
\subref{subfig:uh8} $\overline{\left(U\#H_1^*\right)_{8}}$;
\subref{subfig:ut1t1t1} $\overline{\left(U\#T_1^*\#T_1^*\#T_1^*\right)_{9}}$. %
}
\label{Fig:closure-alternative}
\end{figure}
\item if  $ K\in \mathcal{S}_{2,2} $ and $ K=T_1\#T_1^* $, then $ \overline{K_n}(x)=\left( \overline{T_n} \#T_n\right)(x)$, see \Figs{Fig:altsum} \subref{subfig:alt1};
\item if $ K\in \mathcal{S}_{3,1}\cup\mathcal{S}_{3,2} $, then $ \overline{K_n}(x)=\begin{cases}
\left( \overline{T_{2n}} \#T_{n}\right)(x), & \textnormal{if $ K=T_2\#T_{1}^* $, see \Figs{Fig:altsum} \subref{subfig:alt2}};\\
\left(\overline{T_n}\#T_{2n}\right)(x), &   \textnormal{if $ K=T_1\#T_{2}^* $ or $ K=T_1\#T_{1}^*\#T_{1}^* $,}\\
&\textnormal{see \Figs{Fig:altsum} \subref{subfig:alt3}};
\end{cases} $
\item if $ K\in \mathcal{S}_{3,3} $, then $ \overline{K_n}(x)=\begin{cases}
\left( \overline{L_{n}} \#T_n\right)(x), & \textnormal{if $ K=L_1\#T_{1}^* $, see \Figs{Fig:altsum} \subref{subfig:alt4}};\\
\left(\overline{T_n}\#L_{n}\right)(x), & \textnormal{if $ K=T_1\#L_{1}^* $, see \Figs{Fig:altsum} \subref{subfig:alt5}}.
\end{cases} $
\end{itemize}

\begin{figure}[ht]
\centering
\hspace*{\fill}
\subfigure[]{\includegraphics[width=0.13\linewidth]{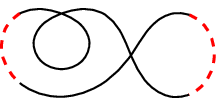}\label{subfig:alt1}}\hfill%
\subfigure[]{\includegraphics[width=0.18\linewidth]{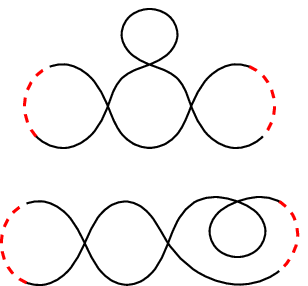}\label{subfig:alt2}}\hfill%
\subfigure[]{\includegraphics[width=0.155\linewidth]{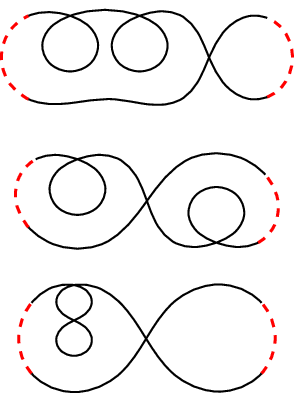}\label{subfig:alt3}}\hfill
\subfigure[]{\includegraphics[width=0.11\linewidth]{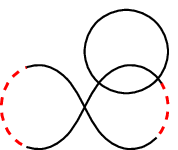}\label{subfig:alt4}}\hfill%
\subfigure[]{\includegraphics[width=0.12\linewidth]{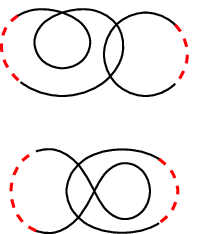}\label{subfig:alt5}}
\hspace*{\fill}
\caption{The arcs at which the connected sums are applied are indicated by the red dashed sections:
\subref{subfig:alt1} $T_1\#T_1^*$;
\subref{subfig:alt2} $T_1\#T_1^*\#T_1^*$ and $T_1\#T_2^*$;
\subref{subfig:alt3} $T_1\#T_1^*$;
\subref{subfig:alt4} $T_1\#L_1^*$;
\subref{subfig:alt5} $L_1\#T_1^*$.%
}
\label{Fig:altsum}
\end{figure}

\paragraph{Results on {\boldmath $ \left( \overline{T_n} \#T_{n}\right)(x) $}:}

\begin{enumerate}
\item Generating polynomial: \begin{equation}
\left( \overline{T_n} \#T_n\right)(x)=\left(x^2+2x+1\right)^{n}+\left(x^2-1\right)(x+1)^{n}.
\end{equation}
\item Generating function:\begin{equation}
\dfrac{1}{1-y(x+1)}\left(x^2+\dfrac{yx(x+1)}{1-y(x^2+2x+1)}\right):=\sum_{n\geq0}^{}\left( \overline{T_n} \#T_n\right)(x)y^n.
\end{equation}
\item Distribution of  $\sigma_{{\textnormal{a}}}(n,k):= \left[x^k\right]\left( \overline{T_n} \#T_n\right)(x)$:  see \Tabs{Tab:dank} \cite[\seqnum{A300192}]{Sloane}.
\begin{equation}
\begin{cases}
\sigma_{{\textnormal{a}}}(0,2)=1;\\
\sigma_{{\textnormal{a}}}(n,0)=0,\ \sigma_{{\textnormal{a}}}(n,1)=n, & n\geq 0;\\
\sigma_{{\textnormal{a}}}(n,k)=\sigma_{{\textnormal{a}}}(n-1,k-1)+\sigma_{{\textnormal{a}}}(n-1,k)\\
\hphantom{\sigma_{{\textnormal{a}}}(n,k)=}+t_{2}(n-1,k-1)+t_{2}(n-1,k),	&k\geq 1,\ n\geq 0.
\end{cases}
\end{equation}
\begin{table}[ht]
\centering
$\begin{array}{c|rrrrrrrrrrrrrrr}
n\ \backslash\ k	&0	&1	&2	&3	&4	&5	&6	&7	&8	&9	&10	&11	&12&13&14\\
\midrule
0	&0	&0	&1	&	&	&	&	&	&	&	&	&	&&&\\
1	&0	&1	&2	&1	&	&	&	&	&	&	&	&	&&&\\
2	&0	&2	&6	&6	&2	&	&	&	&	&	&	&	&&&\\
3	&0	&3	&13	&22	&18	&7	&1	&	&	&	&	&	&&&\\
4	&0	&4	&23	&56	&75	&60	&29	&8	&1	&	&	&	&&&\\
5	&0	&5	&36	&115	&215	&261	&215	&121	&45	&10	&1	&	&&&\\
6	&0	&6	&52	&206	&495	&806	&938	&798	&496	&220	&66	&12	&1&&\\
7 &0 & 7 & 71 & 336 & 987 & 2016 & 3031 & 3452 & 3010 & 2003 & 1001 & 364 & 91 & 14 &1
\end{array}$
\caption{Values of $\sigma_{{\textnormal{a}}}(n,k)$ for $ 0\leq n\leq 7 $ and $ 0\leq k\leq 14$.}
\label{Tab:dbnk}
\end{table}
\begin{itemize}
\item   $ \sigma_{{\textnormal{a}}}(n,1)=n $, the nonnegative integers \cite[\seqnum{A001477}]{Sloane};
\item  $\sigma_{{\textnormal{a}}}(n,2)=\dfrac{3n^2-n+2}{2}$ \cite[\seqnum{A143689}]{Sloane};
\item  $\sigma_{{\textnormal{a}}}(n,n+3)=\binom{2n}{n-3}$ \cite[\seqnum{A002696}]{Sloane};
\item  $\sigma_{{\textnormal{a}}}(n,n+4)=\binom{2n}{n-4}$ \cite[\seqnum{A004310}]{Sloane};
\item  $\sigma_{{\textnormal{a}}}(n,n+5)=\binom{2n}{n-5}$ \cite[\seqnum{A004311}]{Sloane}.
\end{itemize}
\end{enumerate}

\paragraph{Results on {\boldmath $ \left( \overline{T_{2n}} \#T_{n}\right)(x) $}:}
\begin{enumerate}
\item Generating polynomial: \begin{equation}
\left( \overline{T_{2n}} \#T_n\right)(x)=\left(x^3+3x^2+3x+1\right)^{n}+\left(x^2-1\right)(x+1)^{n}.
\end{equation}
\item Generating function:  \begin{equation}
\dfrac{1}{1-y(x+1)}\left(x^2+\dfrac{yx\left(x^2+3x+2\right)}{1-y\left(x^3+3x^2+3x+1\right)}\right):=\sum_{n\geq0}^{}\left( \overline{T_{2n}} \#T_n\right)(x)y^n.
\end{equation}
\item Distribution  of $\sigma_{{\textnormal{b}}}(n,k):= \left[x^k\right]\left( \overline{T_{2n}} \#T_n\right)(x)$: see \Tabs{Tab:denk}.
\begin{equation}
\begin{cases}
\sigma_{{\textnormal{b}}}(0,2)=1;\\
\sigma_{{\textnormal{b}}}(n,0)=0,\ \sigma_{{\textnormal{b}}}(n,1)=2n,\ \sigma_{{\textnormal{b}}}(n,2)=4n^2-n+1, & n\geq 0;\\
\sigma_{{\textnormal{b}}}(n,k)=\sigma_{{\textnormal{b}}}(n-1,k-1)+\sigma_{{\textnormal{b}}}(n-1,k)\\
\hphantom{\sigma_{{\textnormal{b}}}(n,k)=}+t_3(n-1,k-2)+3t_3(n-1,k-1)+2t_3(n-1,k),	&k\geq 2,\ n\geq 0.
\end{cases}
\end{equation}
\begin{table}[ht]
\centering
\resizebox{\linewidth}{!}{%
$\begin{array}{c|rrrrrrrrrrrrrrrr}
n\ \backslash\ k	&0	&1	&2	&3	&4	&5	&6	&7	&8	&9	&10	&11	&12 &13 & 14 &15\\
\midrule
0	&0	&0	&1	&	&	&	&	&	&	&	&	&	&&&&\\
1	&0	&2	&4	&2	&	&	&	&	&	&	&	&	&&&&\\
2	&0	&4	&15	&22	&16	&6	&1	&	&	&	&	&	&&&&\\
3	&0	&6	&34	&86	&129	&127	&84	&36	&9	&1	&	&	&&&&\\
4	&0	&8	&61	&220	&500	&796	&925	&792	&495	&220	&66	&12	&1&&&\\
5 & 0& 10 & 96 & 450 & 1370 &3012 & 5010 & 6436 & 6435 & 5005 & 3003 & 1365 & 455 & 105 & 15 & 1 
\end{array}$
}
\caption{Values of $ \sigma_{{\textnormal{b}}}(n,k)$ for $ 0\leq n\leq 5 $ and $ 0\leq k\leq 15$.}
\label{Tab:denk}
\end{table}
\begin{itemize}
\item   $ \sigma_{{\textnormal{b}}}(n,1)=2n $,  the nonnegative even numbers \cite[\seqnum{A005843}]{Sloane};
\item  $\sigma_{{\textnormal{b}}}(n,2)=4n^2-n+1$ \cite[\seqnum{A054556}]{Sloane}.
\end{itemize}
\end{enumerate}
\paragraph{Results on {\boldmath $ \left( \overline{T_n} \#T_{2n}\right)(x) $}:}
\begin{enumerate}
\item Generating polynomial:   \begin{equation}
\left( \overline{T_n} \#T_{2n}\right)=\left(x^3+3x^2+3x+1\right)^{n}+\left(x^2-1\right)\left(x^2+2x+1\right)^{n}.
\end{equation}
\item Generating function: \begin{equation}
\dfrac{1}{1-y\left(x^2+2x+1\right)}\left(x^2+\dfrac{yx\left(x^2+2x+1\right)}{1-y\left(x^3+3x^2+3x+1\right)}\right):=\sum_{n\geq0}^{}\left(\overline{T_n}\#T_{2n}\right)(x)y^n.
\end{equation}
\item Distribution of $\sigma_{{\textnormal{c}}}(n,k):= \left[x^k\right]\left( \overline{T_n} \#T_{2n}\right)(x)$:  see \Tabs{Tab:dfnk}.
\begin{equation}
\begin{cases}
\sigma_{{\textnormal{c}}}(0,2)=1;\\
\sigma_{{\textnormal{c}}}(n,0)=0,\ \sigma_{{\textnormal{c}}}(n,1)=n,\ \sigma_{{\textnormal{c}}}(n,2)=\dfrac{5n^2-n+2}{2}, & n\geq 0;\\
\sigma_{{\textnormal{c}}}(n,k)=\sigma_{{\textnormal{c}}}(n-1,k-2)+2\sigma_{{\textnormal{c}}}(n-1,k-1)+\sigma_{{\textnormal{c}}}(n-1,k)\\
\hphantom{\sigma_{{\textnormal{c}}}(n,k)=}+t_3(n-1,k-2)+2t_3(n-1,k)+t_3(n-1,k),	&k\geq 2,\ n\geq 0.
\end{cases}
\end{equation}
\begin{table}[ht]
\centering
\resizebox{\linewidth}{!}{%
$\begin{array}{c|rrrrrrrrrrrrrrrr}
n\ \backslash\ k	&0	&1	&2	&3	&4	&5	&6	&7	&8	&9	&10	&11	&12&13&14&15\\
\midrule
0	&0	&0	&1	&	&	&	&	&	&	&	&	&	&&&&\\
1	&0	&1	&3	&3	&1	&	&	&	&	&	&	&	&&&&\\
2	&0	&2	&10	&20	&20	&10	&2	&	&	&	&	&	&&&&\\
3	&0	&3	&22	&70	&126	&140	&98	&42	&10	&1	&	&	&&&&\\
4	&0	&4	&39	&172	&453	&792	&966	&840	&522	&228	&67	&12	&1&&&\\
5 & 0& 5 & 61 & 345 & 1200 & 2871 & 5005 & 6567 & 660 & 5115 & 3047 & 1375 & 456 & 105 & 15 & 1
\end{array}$
}
\caption{Values of $ \sigma_{{\textnormal{c}}}(n,k)$ for $ 0\leq n\leq 5 $ and $ 0\leq k\leq 15$.}
\label{Tab:dfnk}
\end{table}
\begin{itemize}
\item  $ \sigma_{{\textnormal{c}}}(n,1)=n $, the nonnegative integers \cite[\seqnum{A001477}]{Sloane};
\item  $\sigma_{{\textnormal{c}}}(n,2)=\dfrac{5n^2-n+2}{2}$ \cite[\seqnum{A140066}]{Sloane}.
\end{itemize}
\end{enumerate}

\paragraph{Results on {\boldmath $ \left( \overline{L_{n}} \#T_n\right)(x)$}:}
\begin{enumerate}
\item Generating polynomial: 
\begin{equation}
\big( \overline{L_{n}} \#T_{n}\big)(x)=\left(2x^2+4x+2\right)^n+\left(x^2-1\right)\left(x^2+3x+2\right)^n.
\end{equation}
\item Generating function: 
\begin{equation}
\dfrac{1}{1-y\left(x^2+3x+2\right)}\left(x^2+\dfrac{yx(x+1)}{1-y(2x^2+4x+2)}\right):=\sum_{n\geq0}^{}\big( \overline{L_{n}} \#T_{n}\big)(x)y^n.
\end{equation}
\item Distribution of $\sigma_{{\textnormal{d}}}(n,k):= \left[x^k\right]\left( \overline{L_{n}} \#T_{n}\right)(x)$:  see \Tabs{Tab:dink}.
\begin{equation}
\begin{cases}
\sigma_{{\textnormal{d}}}(0,2)=1;\\
\sigma_{{\textnormal{d}}}(n,0)=0,\ \sigma_{{\textnormal{d}}}(n,1)= n2^{n-1}, \ \sigma_{{\textnormal{d}}}(n,2)= 2^{n-3}\left(7n^2-3n+8\right),& n\geq 0;\\
\sigma_{{\textnormal{d}}}(n,k)=\sigma_{{\textnormal{d}}}(n-1,k-2)+3\sigma_{{\textnormal{d}}}(n-1,k-1)+2\sigma_{{\textnormal{d}}}(n-1,k)\\
\hphantom{\sigma_{{\textnormal{d}}}(n,k)=}+w(n-1,k-1)+w(n-1,k),	&k\geq 2,\ n\geq 0.
\end{cases}
\end{equation}

\begin{table}[ht]
\centering
\resizebox{\linewidth}{!}{%
$\begin{array}{c|rrrrrrrrrrrrrrr}
n\ \backslash\ k		 &0		 &1		 &2		 &3		 &4		 &5		 &6		 &7		 &8		 &9		 &10		 &11 	 &12	 &13	 &14\\
\midrule
0	 &0	 &0	 &1	 &	 &	 &	 &	 &	 &	 &	 &	 &	 &	 &	 &\\
1	 &0	 &1	 &3	 &3	 &1	 &	 &	 &	 &	 &	 &	 &	 &	 &	 &\\
2	 &0	 &4	 &15	 &22	 &16	 &6	 &1	 &	 &	 &	 &	 &	 &	 &	 &\\
3	 &0	 &12	 &62	 &133	 &153	 &102	 &40	 &9	 &1	 &	 &	 &	 &	 &	 &\\
4	 &0	 &32	 &216	 &632	 &1047	 &1076	 &707	 &296	 &77	 &12	 &1	 &	 &	 &	 &\\
5	 &0	 &80	 &672	 &2520	 &5550	 &7941	 &7705	 &5133	 &2325	 &695	 &131	 &15	 &1\\
6	 &0	 &192	 &1936	 &8896	 &24612	 &45612	 &59567	 &56106	 &38314	 &18778	 &6432	 &1470	 &210	 &18	 &1
\end{array}$
}
\caption{Values of $ \sigma_{{\textnormal{d}}}(n,k)$ for $ 0\leq n\leq 6 $ and $ 0\leq k\leq 14 $.}
\label{Tab:dink}
\end{table}

\begin{itemize}
\item $ \sigma_{{\textnormal{d}}}(n,1)=n2^{n-1} $ \cite[\seqnum{A001787}]{Sloane};
\item $ \sigma_{{\textnormal{d}}}(n,2n+1)=3n$, the multiples of 3 \cite[\seqnum{A008585}]{Sloane}.
\end{itemize}
\end{enumerate}
\paragraph{Results on {\boldmath $ \left( \overline{T_n} \#L_n\right)(x)$}:}
\begin{enumerate}
\item Generating polynomial:
\begin{equation}
\left( \overline{T_n} \#L_n\right)(x)=\left(2x^2+4x+2\right)^n+\left(x^2-1\right)(2x+2)^n.
\end{equation}
\item Generating function: 
\begin{equation}
\dfrac{1}{1-y(2x+2)}\left(x^2+\dfrac{yx(2x+2)}{1-y(2x^2+4x+2)}\right):=\sum_{n\geq0}^{}\left(\overline{T_n}\#L_{n}\right)(x)y^n.
\end{equation}
\item Distribution of $\sigma_{{\textnormal{e}}}(n,k):= \left[x^k\right]\left( \overline{T_n} \#L_{n}\right)(x)$:  see \Tabs{Tab:djnk}.
\begin{equation}
\begin{cases}
\sigma_{{\textnormal{e}}}(0,2)=1;\\
\sigma_{{\textnormal{e}}}(n,0)=0,\ \sigma_{{\textnormal{e}}}(n,1)=n2^n\  \cite[\seqnum{A036289}]{Sloane}, & n\geq 0;\\
\sigma_{{\textnormal{e}}}(n,k)=2\sigma_{{\textnormal{e}}}(n-1,k-1)+2\sigma_{{\textnormal{e}}}(n-1,k)\\
\hphantom{\sigma_{{\textnormal{e}}}(n,k)=}+2w(n-1,k-1)+2w(n-1,k),	&k\geq 1,\ n\geq 0.
\end{cases}
\end{equation}

\begin{table}[ht]
\centering
\resizebox{\linewidth}{!}{%
$\begin{array}{c|rrrrrrrrrrrrrr}
n\ \backslash\ k		 &0		 &1		 &2		 &3		 &4		 &5		 &6		 &7		 &8		 &9		 &10		 & 11 	 &12\\
\midrule
0	 &0	 &0	 &1	 &	 &	 &	 &	 &	 &	 &	 &	 &	 &\\
1	 &0	 &2	 &4	 &2	 &	 &	 &	 &	 &	 &	 &	 &	 &\\
2	 &0	 &8	 &24	 &24	 &8	 &	 &	 &	 &	 &	 &	 &	 &\\
3	 &0	 &24	 &104	 &176	 &144	 &56	 &8	 &	 &	 &	 &	 &	 &\\
4	 &0	 &64	 &368	 &896	 &1200	 &960	 &464	 &128	 &16	 &	 &	 &	 &\\
5	 &0	 &160	 &1152	 &3680	 &6880	 &8352	 &6880	 &3872	 &1440	 &320	 &32	 &	 &\\
6	 &0	 &384	 &3328	 &13184	 &31680	 &51584	 &60032	 &51072	 &31744	 &14080	 &4224	 &768	 &64
\end{array}$
}
\caption{Values of $ \sigma_{{\textnormal{e}}}(n,k) $ for $ 0\leq n\leq 6 $ and $ 0\leq k\leq 12 $.}
\label{Tab:djnk}
\end{table}

\end{enumerate}

\small \textbf{2010 Mathematics Subject Classifications}:  05A19;  57M25.
\end{document}